\documentclass [11pt]{article}
\usepackage{amsmath}
\usepackage{amssymb}
\usepackage{amscd}
\usepackage{amsthm}
\usepackage{amsopn}
\usepackage{xspace}
\usepackage{verbatim}
\usepackage{amsmath}
\usepackage{amsfonts}
\usepackage{graphicx}
\bibliography{bibfile}

\usepackage{subfigure}
\usepackage{subfigmat}
\usepackage{float}

\newtheorem{theorem}{Theorem}
\newtheorem{lemma}{Lemma}[section]

\newtheorem{proposition}{Proposition}[section]

\newtheorem{assumption}{Assumption}[section]
\newtheorem{definition}{Definition}[section]
\newtheorem{acknowledgment*}{Acknowledgment}
\newtheorem{remark}{Remark}[section]
\numberwithin{equation}{section}

\newcommand{\uu}{\underline{u}}
\newcommand{\urho}{\underline{\rho}}

\newcommand{\be}{\begin{equation}}
\newcommand{\ee}{\end{equation}}
\newcommand{\bd}{\begin{displaymath}}
\newcommand{\ed}{\end{displaymath}}

\newcommand{\eps}{\varepsilon}

\newcommand{\R}{\mathbb R}

\newcommand{\wrho}{\rho_2}

\newcommand{\uw}{\underline{w}}

\renewcommand{\vec}[1]{\boldsymbol{#1}}

\begin{document}
\Large
\begin{center}{\bf Chemotactic systems in the presence of   conflicts: a new functional inequality} \end{center}
\normalsize
\begin{center}   G. Wolansky \\
Department of Mathematics, Technion, Haifa 32000, israel\end{center}
\begin{abstract}
The evolution of a chemotactic system involving a  population of cells attracted to self-produced chemicals  is described by the Keller-Segel system. In
dimension 2, this system demonstrates a balance between the spreading effect of diffusion and the  concentration due to self-attraction. As a result, there exists  a critical "mass" (i.e. total cell's population) above which the solution of this system collapses in a finite time, while below this critical mass there is global existence in time. In particular, sub critical mass leads under certain additional conditions to the existence of steady states, corresponding to the solution of an elliptic Liouville equation.  The existence of this critical mass is related to a functional inequality known as the Moser-Trudinger inequality.

An extension of the Keller-Segel model to several cells populations was considered before in the literature. Here we review some of these results and, in particular, consider the case of conflict between two populations, that is, when  population one attracts  population two, while, at the same time, population two repels population one. This assumption leads to a new functional inequality which generalizes the Moser-Trudinger inequality. As an application of this inequality we derive sufficient conditions for the existence of steady states corresponding to solutions of an elliptic  Liouville system.
\end{abstract}
\section{Introduction}\label{intro}
In this   paper we study  a
 non-local elliptic Liouville  system  in $\Omega$ of the form
\be\label{Ils}\Delta u+M_1\frac{e^{\alpha u\pm\beta w}}{\int_\Omega e^{\alpha u\pm\beta w}}=0 \ \ \ ; \ \ \ \Delta w+M_2\frac{e^{-\gamma w+\beta u}}{\int_\Omega e^{-\gamma w+\beta u}} =0 \  , \  \ \ u=w=0 \ \ \text{on} \ \partial\Omega \   \ee
where $M_i>0$,  all constants $\alpha,\beta,\gamma$ are non-negative and $\Omega$ is a planar bounded domain.  We denote the $+\beta$ case above as the "conflict free" case, while the $-\beta$ is the "conflict" case. The reasoning behind this notation is explained below (see also  Section \ref{pre} and [\ref{[W1]}]).

Our motivation for studying this system is the  non-local parabolic-elliptic system
\be\label{Idelta2=0sys1}\frac{\partial\rho}{\partial t} = \Delta\rho + \nabla\cdot\left[\rho( \mp\beta \nabla w-\alpha\nabla u )\right]   \ \ ; \ \  \Delta u+\rho =0 \ \ ; \ \  \Delta w+M_2\frac{e^{\beta u-\gamma w}}{\int e^{\beta u-\gamma w}}=0    \ee
System (\ref{Idelta2=0sys1}) is defined  on $\Omega\times [0, \infty)$. The no-flux boundary condition for $\rho$ takes the form
 \be\label{Inoflus} \left(\nabla\rho- \alpha\rho\nabla u\mp\beta\rho\nabla w  \right)\cdot \vec{n}=0 \ \  \text{on} \ \partial\Omega\times (0,\infty) \ee
 where $n$ is the normal to $\partial \Omega$. In addition, $u=w=0$ on $\partial\Omega\times (0,\infty)$.
 \par
  In addition,  $\rho,u,w$ satisfy the initial conditions at $t=0$:
$u(,0)=u^0\in \mathbb{H}_0^1(\Omega):=\mathbb{H}_0^1 $, $w(,0)=w^0\in\mathbb{H}_0^1$
 and $\rho(,0):=\rho^0\in \mathbb{L}^1(\Omega):= \mathbb{L}^1$ where $\rho^0\geq 0$  on $\Omega$. In particular, the no-flux boundary condition (\ref{Inoflus}) implies, by a formal application of the divergence theorem, the  conservation of mass:
\be\label{Iid} \int_\Omega\rho(x,t)dx=\int_\Omega\rho^0(x)dx:= M_1>0 \ . \ee
The steady states of (\ref{Idelta2=0sys1}, \ref{Inoflus}, \ref{Iid}) are solutions of (\ref{Ils}) where $\rho=M_1\frac{e^{\alpha u-\beta w}}{\int_\Omega e^{\alpha u-\beta w}}$. \par
 The function $\rho$ corresponds, in the language of chemotaxis [\ref{[KS]},\ref{[H]}], to  the density  of a  population of  organisms (living cells, bacteria, slime molds or, perhaps, crowded human beings  ...) which evolve in time without multiplication and mortality. The individuals  of this population are moving  on the planar domain $\Omega$ under  a combination of random walk  and deterministic drift force along the gradient of  self produced chemicals $u$ and $w$.
 \par
 We remark at this point that the sign of the off diagonal terms in (\ref{Ils}) represents the interaction force between the populations. A positive off diagonal term for a given component represents that the  population corresponding to this component  is rejected from the other population. Thus, $\beta>0$ implies, due the second equation  in (\ref{Ils}), that the second population is rejected from the first one. The choice $+\beta$ in the first equation implies that the first population is rejected from the second one as well, so there is no conflict. On the other hand, the choice $-\beta$ implies that the first population is {\it attracted} to the second one (while the second one is still rejected by the first). This unhappy situation  is the origin of conflict of interests between the two populations.
\subsection{The conflict Free case}
In [\ref{[CSW]}], the general version of Liouville system (\ref{Ils})   was considered
\be\label{lsn} \Delta u_i + \frac{M_i}{\int_\Omega e^{\sum_{j=1}^na_{ij}u_j}}e^{\sum_{j=1}^na_{ij} u_j}=0 \ \ \ \   \  \ee
of $n\geq 2$ populations.
This system is defined on a bounded domain $\Omega\subset\R^2$, where $\ u_i=0 \ \text{on} \ \ \partial\Omega$ and
where $M_i>0$ are constants, $1\leq i\leq n $. The coefficients $\{a_{i,j}\}$ are assumes to be non-negative and  $a_{i,j}=a_{j,i}$.

Let $J\subset \{1, \ldots  n\}$, and
$$ \Lambda_J(M_1, \ldots M_n):= 4\pi\sum_{i\in J}M_i -\frac{1}{2}\sum_{i,j\in J}a_{i,j}M_iM_j  \ . $$
Theorem 1.1 in [\ref{[CSW]}] implies that a  sufficient condition for the existence of a solution of (\ref{lsn}) is the inequalities
\be\label{mineq} \Lambda_J(M_1, \ldots M_n)>0 \ \ \ \text{for any} \ \ J\subseteq \{1, \ldots n\}, \ J\not=\emptyset \ . \ee
Theorem 1.2 of the same paper  deals with  {\em radial} solutions of (\ref{lsn}) under the assumption that $\Omega$ is a disk
 in $\R^2$. It follows that in that case  the same result holds even if we give up the condition of non-negative {\em off diagonal} elements $a_{i,j}$, $i\not=j$.  The diagonal elements $a_{i,i}$ are still assumed to be non negative.
\par
Note that the conflict free case $+\beta$ in (\ref{Ils}) corresponds to a symmetric matrix $\{a_{i,j}\}$ in (\ref{lsn}) where $n=2$. However, the presence of the negative coefficient $-\gamma$ in the second equation of (\ref{Ils}) violates the non-negative diagonal assumption of Theorem 1.2 in [\ref{[CSW]}]. Still,  the proof in [\ref{[CSW]}]  can be extended to this case as well, provided   condition   (\ref{mineq} ) is replaced by
\be\label{mineqm} \Lambda_I(m_1, \ldots m_n)>0 \   \ \ ,  \ 0<m_i\leq M_i \  \text{for any}  \ i\in I:=\{1, \ldots n\}  \ . \ee
Note that (\ref{mineq}) implies (\ref{mineqm}) if the diagonal elements are  non-negative ($a_{i,i}\geq 0$).


An application of (\ref{mineqm}) to  (\ref{Ils}) implies  the condition
\be\label{dcond} 0<M_1<\frac{8\pi}{\alpha}, \ \text{and}\ \
 4\pi(M_1+m)-\frac{\alpha}{2}M_1^2+\frac{\gamma}{2}m^2-\beta M_1m>0 \ \ \forall m\in (0, M_2) \ . \ee
 If $\beta<0$ then the only condition is $M_1<8\pi/\alpha$ and $M_2<\infty$. If $\beta>0$ then  
we distinguish three cases:
\begin{description}
\item{a)}  $2\beta<\alpha$ \ \ $\Rightarrow$ $M_1\in(0,8\pi/\alpha)$, $M_2>0$.
\item{b)} $2\beta\geq \alpha$, $\gamma=0$ $\Rightarrow$
\be\label{onlycond}4\pi(M_1+M_2)-\frac{\alpha}{2}M_1^2+\frac{\gamma}{2}M_2^2-\beta M_1M_2>0\ee
is the only condition.
\item{c)} $2\beta\geq \alpha$, $\gamma>0$ $\Rightarrow$ either $M_1\in(0,\underline{M}_1)$ and $M_2>0$ or $M_1\in(\underline{M}_1, 8\pi/\alpha)$ and
(\ref{onlycond}). Here $M_1=\underline{M}_1$ is the vertical asymptote to the hyperbolic branch of
$$4\pi(M_1+M_2)-\frac{\alpha}{2}M_1^2+\frac{\gamma}{2}M_2^2-\beta M_1M_2=0$$
 in the positive quadrate $M_1, M_2>0$.
\end{description}
The domains in these three cases are demonstrated by the shaded areas in Fig. \ref{fig1}
\begin{figure}
            \begin{subfigmatrix}{3}
                        \subfigure[]{\includegraphics[width=0.32\textwidth]{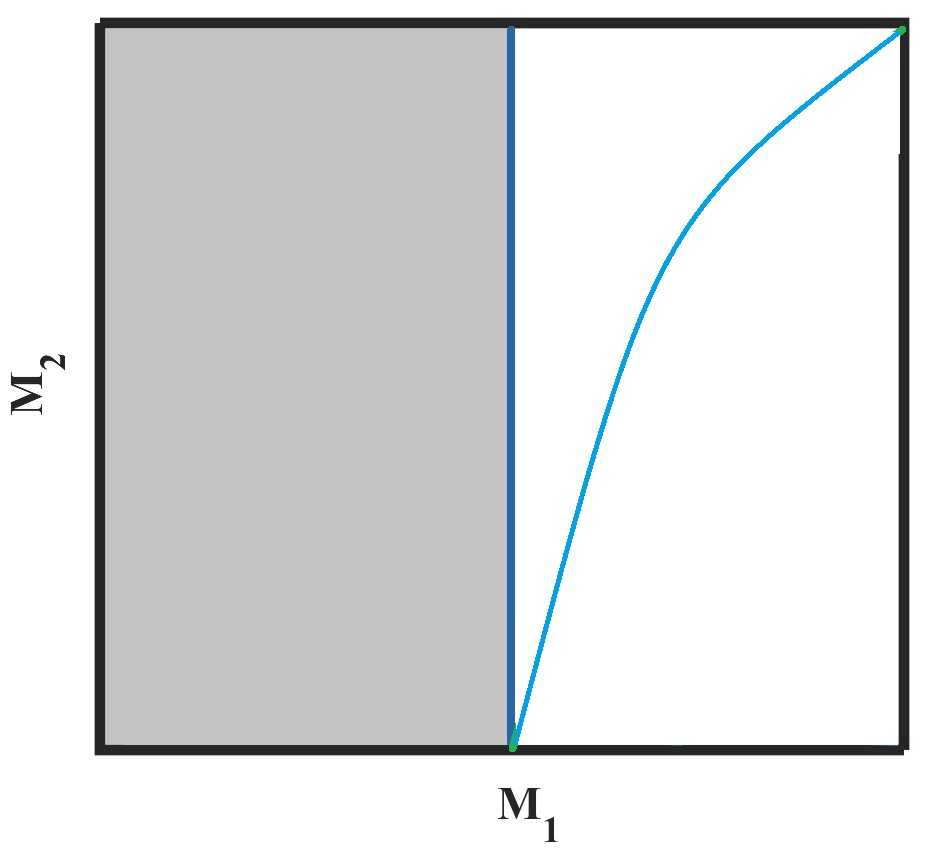}}
                        \subfigure[]{\includegraphics[width=0.32\textwidth]{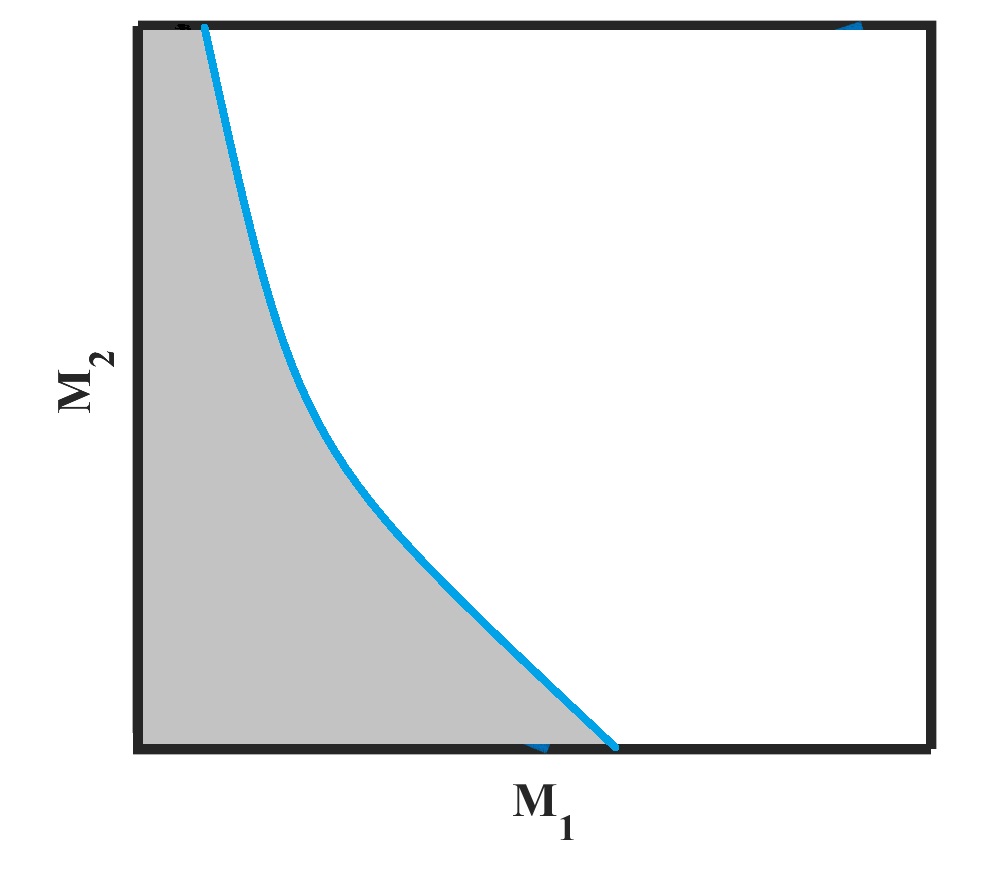}}
                         \subfigure[]{\includegraphics[width=0.32\textwidth]{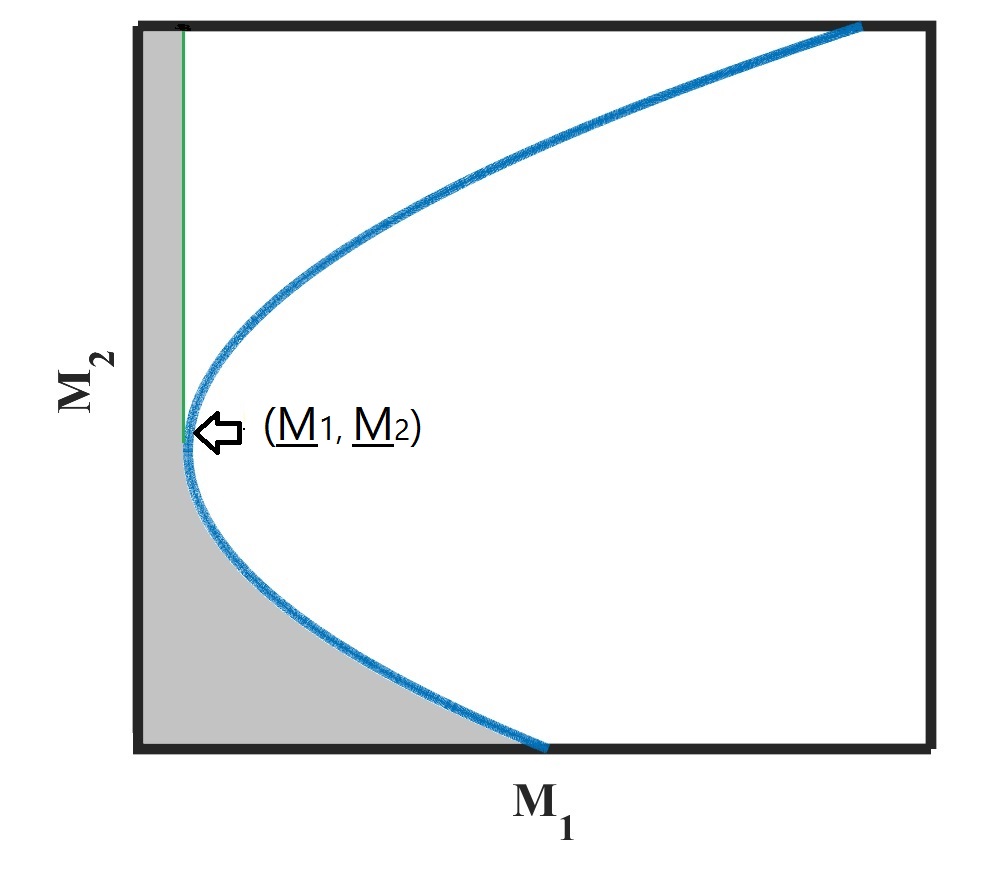}}
            \end{subfigmatrix}
           \caption{The 3 cases of conflict free chemotaxis for $\beta>0$: $\beta<\alpha/2$ (a),  $\beta>\alpha/2$, $\gamma=0$ (b) and  $\beta>\alpha/2$,   $\gamma>0$ (c). The vertical heavy line represents the critical mass $M_1=8\pi/\alpha$.}
            \label{fig1}
\end{figure}
\subsection{The case of conflict}\label{coc}
The main result of this paper is referred to the case of conflict, i.e  ($-\beta$) in (\ref{Ils}). We also assume $\beta>0$ from now on.
Let
$$ \Lambda(M_1, M_2):= 2(M_1-M_2)-\frac{\alpha M_1^2}{4\pi}+\frac{\beta M_1M_2}{2\pi} -\frac{\gamma M_2^2}{4\pi}\ . $$
\begin{theorem}\label{T1}
For any choice of $\alpha>0$, $\beta\geq 0,\gamma\geq 0$ there exists a  solution of (\ref{Ils}) in the conflict case for any $0<M_1< 8\pi/\alpha$, $0\leq M_2<\infty$.
  \par
  If, moreover,  $\beta > \alpha/2$  and  $\Omega$ is a disk in $\R^2$, then a radial solution exists if  
\begin{description}
\item{i)} \be\label{hypelips} \Lambda(M_1, M_2) >0 \ , \ee
and  $M_1<\underline{M}$ where  $\underline{M}$ is determined by the larger root of
\be\label{barMdef}\Lambda\left(\underline{M}, \frac{4\pi}{\gamma} \left(\frac{2\beta}{\alpha}-1\right)\right)=0\ee
if $\gamma>0$, and $\underline{M}=\infty$ if $\gamma=0$.
\item{ii)} If $(M_1,M_2)$ satisfies (i) then  the solution exists also for all $(M_1, M)$ where $M>M_2$.
\end{description}

\end{theorem}
The proof of Theorem \ref{T1} follows from Proposition \ref{propcrit}, Proposition \ref{pro1} and Theorem~\ref{radbd} in section \ref{Mainresults}.
 In Fig.  \ref{fig2} we sketch in gray  the solvability domain of (\ref{Ils}) in the   $(M_1, M_2)$ parameters for $\beta > \alpha/2$:
 \begin{figure}[H]
            \begin{subfigmatrix}{3}
                        \subfigure[]{\includegraphics[width=0.42\textwidth]{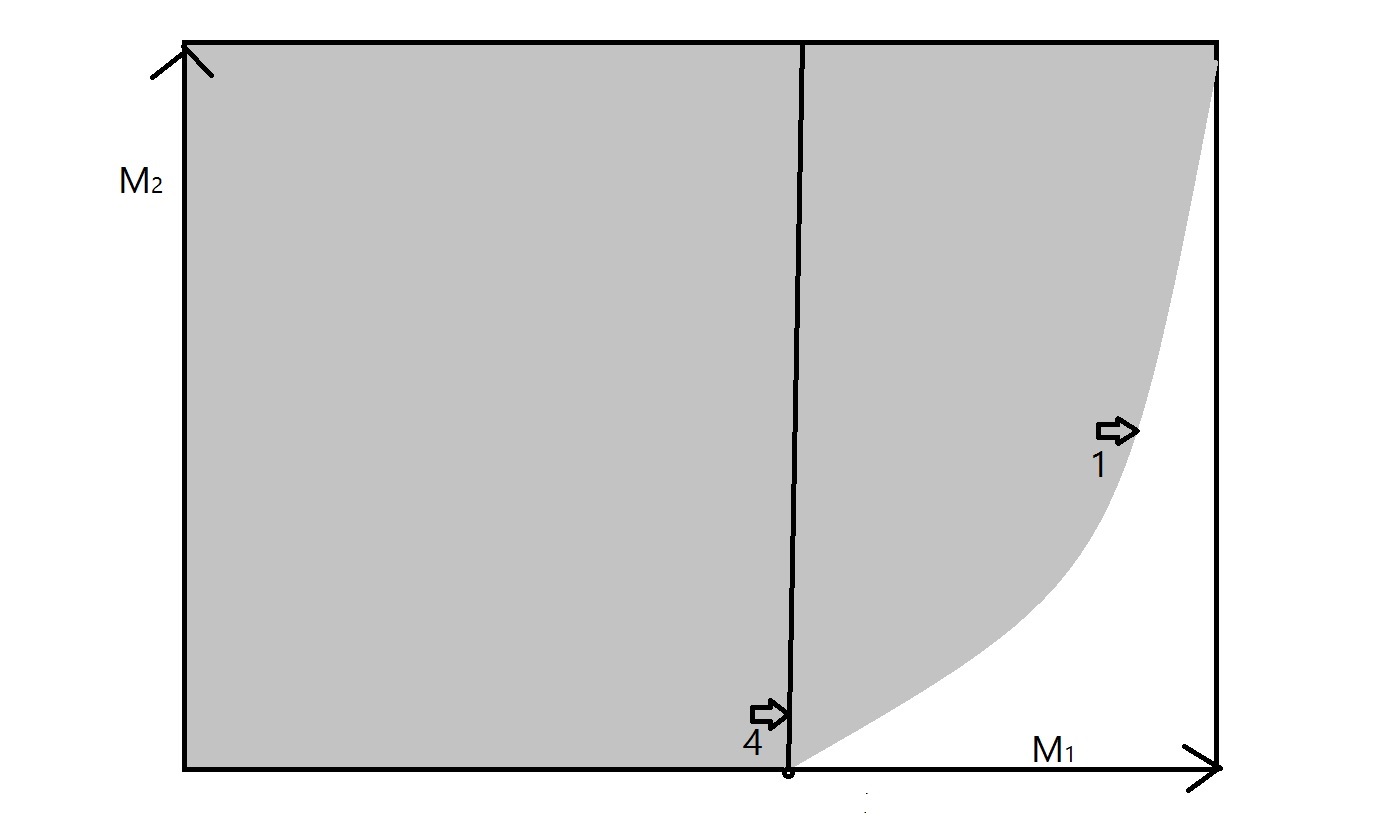}}
                        \subfigure[]{\includegraphics[width=0.42\textwidth]{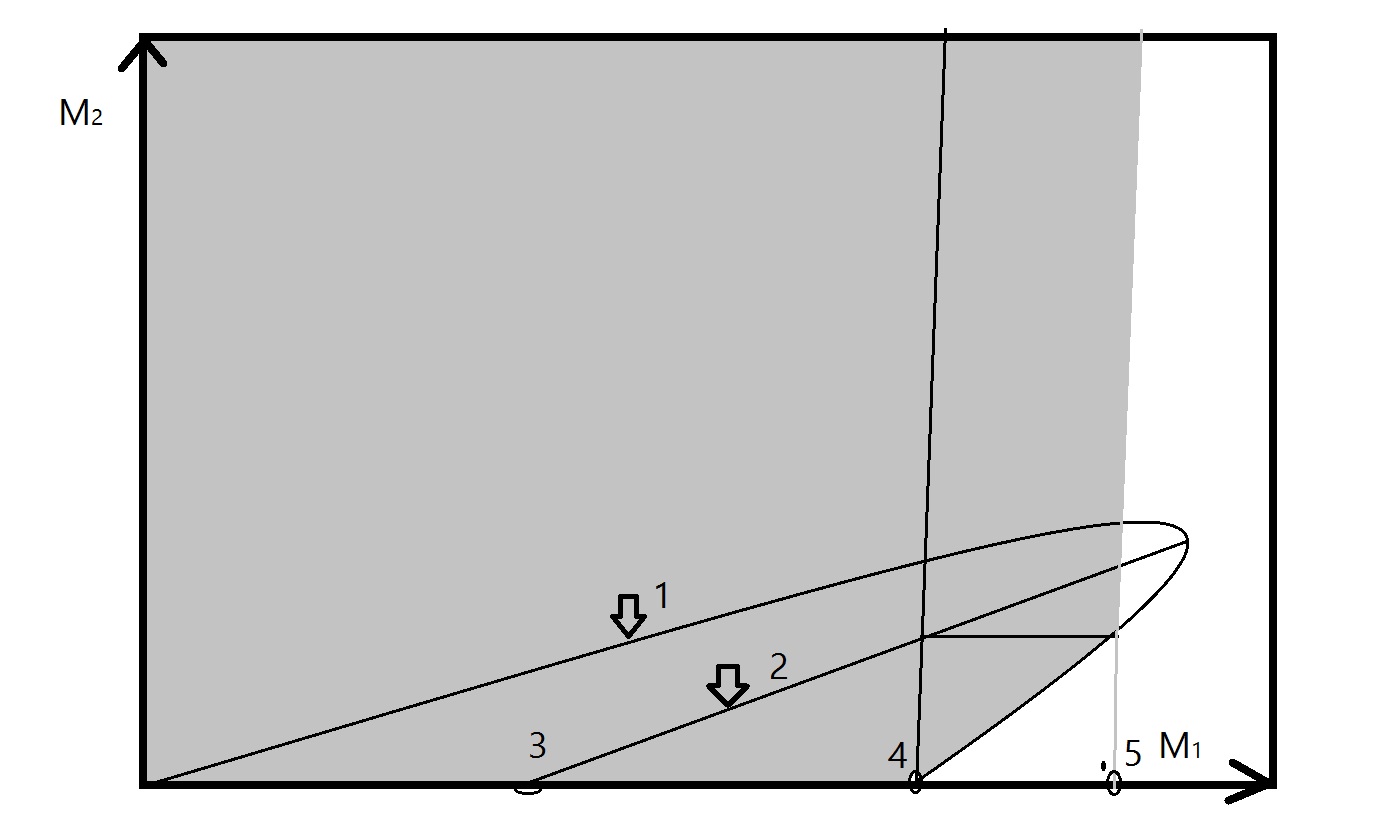}}
            \end{subfigmatrix}
           \caption{The 2 cases of conflict for $\beta>\alpha/2$:   $\gamma=0$ (a), and  $\gamma>0$ (b). \newline
           (1): The curve $\Lambda=0$, (2):$\beta M_1-\gamma M_2=4\pi$, (3): $M_1= 4\pi/\beta$, (4): $M_1=8\pi/\alpha$, (5): $M_1=\underline{M}$}
            \label{fig2}
\end{figure}
\begin{description}
\item {[a]} $\gamma=0$. In that case the domain of solvability coincide with $\Lambda>0$ which is  below the parabola.
\item{[b]}  $\gamma>0$. Here the curve  $\Lambda=0$ is  a quadratic curve (either an ellipse or an hyperbola)  and the solvability domain is bounded by from the right by  $\underline{M}$ on the $M_1$ axis.
\end{description}
\begin{remark} We do not know if the conditions of Theorem \ref{T1} are optimal. However, Theorem~\ref{blow} and the remark below Theorem \ref{radbd} in Section \ref{Mainresults} suggest that this may be  the case, at least if $\gamma=0$.
\end{remark}
\subsection{Structure of the paper} In section \ref{Rev} we review the Free energy method for chemotactic system of a single component, the connection with Moser-Trudinger inequality and its relation with the parabolic and elliptic Liouville equation.
\par
In section \ref{pre} we extend the discussion to chemotactic systems of two components, consider 3 limit cases and the associated Free Energies. The elliptic Liouville systems  for two components are derived in both conflict/noconclict cases.
\par
From  section \ref{seccoc} forward we concentrate in the case of conflict for two component chemotaxis. We discuss the solution of the Liouville system
as steady states of the chemotactic system and its stability under the 3 limit cases. In sections  \ref{secobj}, \ref{Mainresults} we describe the main objectives of this paper and its main results summarized in Theorems \ref{blow} and \ref{radbd}. The most technical part of this paper is the proof of Theorem  \ref{radbd}, given in Section \ref{secprof}.

\subsection{Notations and standing assumptions}\label{notations}
\begin{enumerate}
\item $\Omega\subset \R^2$ is an open, bounded domain.
\item $\partial\Omega$ is the boundary of $\Omega$. We assume that $\partial\Omega$ is  $C^2$ regular.
\item $u\in \mathbb{H}_0^1(\Omega)$ iff $\int_\Omega |u|^2+\int_\Omega |\nabla u|^2 <\infty$ and admits zero trace on $\partial\Omega$.
\item $ \Gamma_M:= \{ \rho\in \mathbb{L}^1(\Omega), \ \rho\geq 0 \ \text{a.e. \ on} \ \Omega, \ \ \int_\Omega\rho\ln\rho < \infty \ , \int_\Omega\rho=M \}$.
    \item $\Delta^{-1}$ is the Green function of the Dirichlet Laplacian on $\Omega$.
\end{enumerate}
\section{Review of Chemotactic systems and Free energy}\label{Rev}
 Define  $F^{M}$ on $\Gamma_{M_1}\times \mathbb{H}_0^1$ as
 $$ F^{M}(\rho,w):= \int_\Omega\rho\ln\rho +\frac{\alpha}{2}\int_{\Omega} \rho \Delta^{-1}\rho +\left[ \frac{\gamma}{2}\int_\Omega|\nabla w|^2 + M\ln\left(\int e^{-\gamma w-\beta \Delta^{-1}(\rho)}\right) \right] \ . $$
 Noting $u=-\Delta^{-1}\rho$,
 it follows that   (\ref{Idelta2=0sys1}) can be written as
 \be\label{IgradF}\frac{\partial\rho}{\partial t} =\nabla\cdot\left( \rho\nabla\frac{\delta F^{M_2}}{\delta \rho}\right) \ \ , \ \ \frac{\delta F^{M_2}}{\delta w}=0 \  \ \text{on} \ \Omega\times (0, \infty) \ee
 while (\ref{Inoflus}) is equivalent to $\nabla\frac{\delta F^{M_2}}{\delta \rho}\cdot \vec{n}=0$ on $\partial\Omega$.
A formal integration by parts yields
 \be\label{Imonotone}\frac{d}{dt} F^{M_2}(\rho(,t), w(,t))= -\int\rho\left|\nabla\frac{\delta F^{M_2}}{\delta\rho}\right|^2\ee
so $F^{M_2}$ is monotone non-increasing along solutions of (\ref{Idelta2=0sys1}).

 From the representation (\ref{IgradF}) it follows that any solution of (\ref{Ils}) corresponds to a critical point of $F^{M_2}$ on $\Gamma_{M_1}\times \mathbb{H}_0^1$. In particular, the monotonicity (\ref{Imonotone}) suggests that  local minimizers of $F^{M_2}$ on this domain correspond to  stable steady states of (\ref{Idelta2=0sys1}). Thus, the question regarding the    bound from below  of $F^{M_2}$ on $\Gamma_{M_1}\times \mathbb{H}_0^1$  is interesting in that respect, as it is a necessary condition for the existence of a global minimizer on this domain. This global minimizer is, evidently, a critical point, and thus a steady state of (\ref{Idelta2=0sys1}).
 \par
 If we substitute $\gamma=\beta=0$ in $F^M$ we get, up to an irrelevant constant,   the {\it Free Energy functional}
 \be\label{IF} \rho\in\Gamma_M \mapsto F(\rho):= \int_\Omega\rho\ln\rho +\frac{\alpha}{2}\int_\Omega \rho \Delta^{-1}(\rho) \ . \ee
 This functional  is monotone non-increasing along solutions  of   the parabolic-elliptic Keller-Segel system for chemotaxis  of a single component [\ref{[S]},\ref{[W1]},\ref{[W2]}, \ref{[BCC]}...] (see also [\ref{[W3]},\ref{[W4]}] for application to self-gravitating systems)
 \be\label{Isingle}\frac{\partial\rho}{\partial t} = \Delta\rho - \nabla\cdot\left[\rho(\alpha\nabla u )\right]   \ \ ; \ \  \Delta u+\rho =0  \ \ ; \ \ \int_\Omega\rho^0=M \ .  \ee
 Note that (\ref{Isingle}) is obtained from the substitution $\gamma=\beta=0$ in (\ref{Idelta2=0sys1}). This   can be written as
 \be\label{IgradF0}\frac{\partial\rho}{\partial t} =\nabla\cdot\left( \rho\nabla\frac{\delta F}{\delta \rho}\right) \ \text{on} \ \Gamma_M\   .  \ee
 The bound from below of $F$ on $\Gamma_M$ for $M\leq 8\pi/\alpha$ follows from the logarithmic HLS inequality [\ref{[B]}, \ref{[CL]}]:
    $$\forall \rho\in\mathbb{L}\ln\mathbb{L}(D) \ , \ \int_{D}|\rho|\ln|\rho| + (4\pi)^{-1}\int_{D}\int_{D}\rho(x)\ln|x-y|\rho(y)dxdy >-C(D) \  \ \ ,  \ \  \|\rho\|_1=1$$
    for functions in a  two dimensional bounded domain $D$.  Using  scaling and  taking into account that $\Delta^{-1}(x,y)\approx(2\pi)^{-1}\ln|x-y|$, up to lower order terms, imply the bound from below on $\Gamma_M$ for $M\leq 8\pi/\alpha$.

  This  is a key inequality for the proof of global existence of (\ref{Isingle}) for $M<8\pi/\alpha$ as well as the existence of solution to the nonlocal Liouville equation
 \be\label{Iles} \Delta u + \frac{M}{\int e^{\alpha u}}e^{\alpha u}=0, \ \ M<8\pi/\alpha \ee
 in a bounded domain  $\Omega$ [\ref{[W2]}, \ref{[KS1]},\ref{[L]}, \ref{[Wa]}].
 \par
 The parabolic elliptic Keller-Segel (\ref{Isingle}) is a limiting case of the parabolic parabolic system  [\ref{[CC]}]
 \be\label{Isinglee}\delta\frac{\partial\rho}{\partial t} = \Delta\rho - \nabla\cdot\left[\rho(\alpha\nabla u )\right]   \ \ ; \ \  \eps\frac{\partial u}{\partial t}=\Delta u+\rho =0  \ \ , \int_\Omega\rho^0=M  , \ \ u^0\in \mathbb{H}_0^1 \ee
 where $\delta=1$ and $\eps=0$. Another, less known limit of (\ref{Isinglee}) [\ref{[W2]}, \ref{[KS1]}]  is $\eps=1, \delta=0$:
 \be\label{Iless} \frac{\partial u}{\partial t}=\Delta u +\frac{M}{\int e^{\alpha u}}e^{\alpha u} \ . \ee
 We observe that (\ref{Iless}) is itself a gradient descend system on $\mathbb{H}_0^1$ of the form
 $$ \frac{\partial u}{\partial t}=-\alpha^{-1}\frac{\delta \overline{H}^M}{\delta u}$$
 where
 $$ u\in \mathbb{H}_0^1 \mapsto \overline{H}_\alpha^M(u):= \frac{\alpha}{2}\int_\Omega|\nabla u|^2 -M\ln\left(\int_\Omega e^{\alpha u}\right) \ . $$
 A simple scaling shows that the bound from below  of $\overline{H}^M$ on $\mathbb{H}_0^1$ where $M\leq 8\pi/\alpha$ follows from the Moser-Trudinger inequality
 \be\label{trudin}\frac{1}{2}\int_\Omega |\nabla u|^2 -8\pi\ln\left(\int_\Omega e^u\right) > -C\ee
 for any $u\in \mathbb{H}_0^1$ [\ref{[RS]}, \ref{[SW]}...]. This gives an alternative proof for the existence of solution to (\ref{Iles}) for $M<8\pi/\alpha$, as well as the global (in time) existence of (\ref{Iless}) under the same condition [\ref{[Bi]}].
 \par
 Motivated by the above, we consider in this paper the condition for bound from below of the functional $F^{M_2}$ on $\Gamma_{M_1}\times \mathbb{H}_0^1$. Note that for $M_2=0$ and $\gamma=0$,  $F^{M_2}$ is just the Free Energy $F$ (\ref{IF}).
 \par
 It follows, then, that  a new inequality for $F^{M_2}>-C$ on $\Gamma_{M_1}\times \mathbb{H}_0^1$ is a generalization of the Logarithmic HLS inequality for the case $M_2=0$. Note also that if $\gamma=0$, $M>0$ then the  last two nonzero terms of $F^M(-\Delta w)$ is, by integration by parts, just
 $$-\frac{\alpha}{2}\int |\nabla w|^2  + M\ln\left(\int e^{\beta w}\right)  $$
 which is related to the Moser-Trudinger inequality (with opposite sign, however).
 In fact, it is known that the Moser-Trudinger and Logarithmic HLS inequalities are equivalent. To see this, consider
  $$ (\rho,u)\in\Gamma_{M_1}\times \mathbb{H}_0^1 \mapsto H(\rho,u):= \int_\Omega \rho\ln\rho +\frac{\alpha}{2}\int_\Omega |\nabla u|^2 -\alpha\int_\Omega \rho u \  \ .  $$
and note that
 $$ F(\rho)=\inf_{u\in \mathbb{H}_0^1} H(\rho,u) \ \ ; \ \ \  \overline{H}_\alpha^M(u):= \inf_{\rho\in\Gamma_M} H(\rho,u) \ , $$
 so both logarithmic HLS and Moser-Trudinger inequalities follow from the bound $H(\rho,u)>-C$ for $(\rho,u)\in \Gamma_{8\pi/\alpha}\times \mathbb{H}_0^1$.
 \par
  It is also interesting to note that  $H$ induces a gradient descend flow for  the parabolic-parabolic Keller-Segel equation (\ref{Isinglee})
  via
  $$ \delta\frac{\partial \rho}{\partial t}=\nabla\cdot\left(\rho \nabla\frac{\delta H}{\delta \rho}\right) \ \ ; \ \ \
 \eps\frac{\partial u}{\partial t}=-\alpha^{-1}\frac{\delta H}{\delta u}$$
 and that (\ref{Isingle}) (resp. (\ref{Iless})) are singular limits of (\ref{Isinglee}) for $\eps=0$ (resp. $\delta=0$).

\section{Multi-Component Chemotactic Systems}\label{pre}
The general system of Chemotaxis for two components is a special case of the  system of $n$ populations [\ref{[W1]}]:
\be\delta_i\label{2sys}\frac{\partial\rho_i}{\partial t} = \sigma_i\Delta\rho_i + \nabla\cdot\left[\rho_i(a_{ii}\nabla u_i +  a_{ij} \nabla u_j)\right]   \ee
where $(i,j)\in \{1,2\}$ $i\not= j$, and
\be\label{sens} \eps\frac{\partial u_i}{\partial t} = b_i\Delta u_i+\rho_i  \ , \ i=1,2 \ee
where $\sigma_i>0, b_i>0$ are the diffusion coefficients, $\delta_i, \eps>0$ and  $a_{i,j}$ are constants.
Eq. (\ref{2sys}, \ref{sens}) are defined on $\Omega\times [0,\infty)$ where $\Omega\subset\R^2$, $u_i$ are subjected to Dirichlet boundary condition $u_i=0$ on $\partial\Omega\times [0,\infty)$ and initial data
\be\label{idu} u_i(,0)=u^0_i\in \mathbb{H}_0^1(\Omega) \ . \ee
 $\rho_i$ satisfy the no-flux boundary conditions
\be\label{noflux} \vec{n}\cdot\left\{\sigma_i\nabla\rho_i + \left[\rho_1(a_{ii}\nabla u_i +  a_{ij} \nabla u_j)\right]\right\}=
0\ee
 on $\partial\Omega\times[0,\infty)$, where $\vec{n}$ is the normal to $\partial\Omega$. In addition,  $\rho_i$ satisfy the initial conditions at $t=0$:
$\rho_i(, 0)=\rho^0_i$,
where $\rho^0_i\in \mathbb{L}^1(\Omega)$ and $\rho^0_i\geq 0$  on $\Omega$. In particular, the no-flux boundary conditions imply, by a formal application of the divergence theorem, the  conservation of mass:
\be\label{id} \int_\Omega\rho_i(x,t)dx=\int_\Omega\rho^0_i(x)dx:= M_i \ . \ee
\par
The functions $\rho_i$ correspond to the densities  of the {\it two}  populations of  organisms  which evolve with time without multiplication and mortality. The individuals  of these populations are moving  on the planar domain $\Omega$ by a combination of random walk (corresponding to the diffusion coefficients $\sigma_i$), and deterministic drift forces along the gradient of  self produced chemicals $u_i$. \par
 Five of the constants in (\ref{2sys}) can be eliminated  by scaling   $\rho_1$, $\wrho$, $u_1, u_2$ and the time $t$. In particular, we can assume, without loosing  generality, that  $\sigma_1=\sigma_2=b_1=b_2=1$ and that  $a_{12}=\pm a_{21}:= \beta$.     Let  $a_{11}:= -\alpha, a_{22}=\gamma$.

\begin{assumption}
$\alpha\geq 0$ (self-attractive first population), $\gamma\geq 0$ (self repulsive second population) as well as $\beta>0$ (first population is rejected by the second one).
\end{assumption}
 We get
 \be\label{sys1}\delta_1\frac{\partial\rho_1}{\partial t} = \Delta\rho_1 + \nabla\cdot\left[\rho_1( \beta \nabla u_2-\alpha\nabla u_1 )\right]   \ \ \ ; \ \ \
 \delta_2\frac{\partial\wrho}{\partial t} =\Delta\wrho +\theta \nabla\cdot\left[\wrho(  \beta \nabla u_1 +\theta\gamma\nabla u_2)\right] \ee
 \be\label{sens0}\eps\frac{\partial u_i}{\partial t}=\Delta u_i+\rho_i  \ , \ \ i=1,2 \ . \ee

  Here
 $\theta\in\{-1,1\}$ corresponds  to the choice of sign in $a_{21}=\pm\beta$.  The case $\theta=1$ is  the {\it conflict free} case studied in [\ref{[W1]}].
 In that case  the second  population is rejected by the first  one, so both population has the same attitude to each other (mutual rejection, in that case).
\par
Let us define
  $H_\theta: \Gamma_{M_1}\times\Gamma_{M_2}\times (\mathbb{H}_0^1)^2\rightarrow \R\cup\{\infty\}$: \\
$H_\theta(\rho_1, \rho_2, u_1, u_2):=$
 \begin{multline} \int\rho_1\ln\rho_1 +\theta \int\rho_2\ln\rho_2 + \\  \frac{\alpha}{2}\left( \int_\Omega| \nabla u_1|^2-2\int_\Omega u_1 \rho_1 \right) - \frac{1}{2}\theta\gamma\left( \int_\Omega |\nabla u_2|^2-2\int_\Omega u_2 \rho_2\right) \\ -\beta\left(\int_\Omega \nabla u_1\cdot\nabla u_2 -\int_\Omega \rho_1 u_2 - \rho_2 u_1 \right)\end{multline}
 The system (\ref{sys1}, \ref{sens0})  subject to initial data  (\ref{idu}, \ref{id}) takes the form
 \be\label{grs}\delta_1\frac{\partial\rho_1}{\partial t} = \nabla\cdot\left(\rho_1\nabla\frac{\delta H_\theta}{\delta\rho_1}\right) \ \ \  ; \ \ \   \delta_2 \frac{\partial\wrho}{\partial t} =\theta \nabla\cdot\left(\wrho\nabla\frac{\delta H_\theta}{\delta\wrho}\right) \ . \ee
 \be\label{qrs} \eps\frac{\partial u_1}{\partial t} = \frac{1}{\beta^2+\alpha\gamma\theta}\left( \beta\frac{\delta H_\theta}{\delta u_2}-
 \theta\gamma\frac{\delta H_\theta}{\delta u_1}\right)
 \ \ \ , \ \  \eps\frac{\partial u_2}{\partial t} =\frac{1}{\beta^2+\alpha\gamma\theta}\left( \beta\frac{\delta H_\theta}{\delta u_1}
 +\alpha \frac{\delta H_\theta}{\delta u_2}\right) \ . \ee


 \subsection{Limit cases}
 \begin{description}
 \item{i)}  The limit $\eps=0$ of system  (\ref{grs}, \ref{qrs}) is reduced into the parabolic-elliptic system (\ref{sys1}) where (\ref{sens0}) is replaced by
 \be\label{eps=0sens0}\Delta u_i+\rho_i =0  \ , \ \ i=1,2 \ . \ee
 If we substitute $u_i=-\Delta^{-1}(\rho_i)$  in $H_\theta$ we get
 \be\label{uHdef}\underline{H}_\theta(\rho_1, \rho_2):= \int\rho_1\ln\rho_1 +\theta \int\rho_2\ln\rho_2 + \frac{\alpha}{2}\int\rho_1 \Delta^{-1}(\rho_1) - \frac{\theta\gamma}{2}\int\rho_2 \Delta^{-1}(\rho_2)-\beta\int\rho_2\Delta^{-1}(\rho_1) \ . \ee
Then,  (\ref{sys1},\ref{eps=0sens0}) takes the form
\be\label{grs1}\delta_1\frac{\partial\rho_1}{\partial t} = \nabla\cdot\left(\rho_1\nabla\frac{\delta \underline{H}_\theta}{\delta\rho_1}\right) \ \ \  ; \ \ \   \delta_2\frac{\partial\rho_2}{\partial t} =\theta \nabla\cdot\left(\wrho\nabla\frac{\delta \underline{H}_\theta}{\delta\wrho}\right) \ . \ee
\item{ii)} The limit $\eps=\delta_2=0$, $\delta_1=1$.  Substitute $\delta_2=0$ in (\ref{sys1}) and integrate
to obtain $\rho_2=M_2e^{-\theta\beta u_1-\gamma u_2}/\int e^{-\theta\beta u_1+\gamma u_2}$. Hence
\be\label{delta2=0sys1}\frac{\partial\rho_1}{\partial t} = \Delta\rho_1 + \nabla\cdot\left[\rho_1( \beta \nabla u_2-\alpha\nabla u_1 )\right]   \ \ ; \ \  \Delta u_1+\rho_1 =0 \ \ ; \ \  \Delta u_2+M_2\frac{e^{-\theta\beta u_1-\gamma u_2}}{\int e^{-\theta\beta u_1-\gamma u_2}}=0   \ . \ee
 Let us define now $F_\theta^{M}$ on $\Gamma_{M_1}\times \mathbb{H}_0^1$ as
 $$ F_\theta^{M}(\rho,w):= \int\rho\ln\rho +\frac{\alpha}{2}\int \rho \Delta^{-1}(\rho) -\theta\left[ \frac{\gamma}{2}\int|\nabla w|^2 + M\ln\left(\int e^{-\gamma w+\theta\beta \Delta^{-1}(\rho)}\right) \right]$$
 Then, with $\rho_1:=\rho$, $u_2:=w$,  (\ref{delta2=0sys1}) can be written as
   \be\label{FM2}\frac{\partial\rho}{\partial t} =\nabla\cdot \rho\nabla\left(\frac{\delta F_\theta^{M_2}}{\delta\rho}\right) \ \ \ ; \ \  \frac{\delta F_\theta^{M_2}}{\delta w}=0 \ . \ee
\item{iii)} The limit  $\delta_1=\delta_2=0$ for  (\ref{grs}, \ref{qrs}) is reduced into
\be\label{qrsexp}\frac{\partial u_1}{\partial t}=\Delta u_1+M_1\frac{e^{\alpha u_1-\beta u_2}}{\int_\Omega e^{\alpha u_1-\beta u_2}} \ \ \ ; \ \ \ \frac{\partial u_2}{\partial t}=\Delta u_2 + M_2\frac{e^{-\gamma u_2-\theta\beta u_1}}{\int_\Omega e^{-\gamma u_2-\theta\beta u_1}}  \ . \ee

 If we substitute (\ref{ls0}) in $H_\theta$ and apply integration by parts, we get
 $\overline{H}^{M_1,M_2}_\theta(u_1, u_2)+M_1\ln M_1 +\theta M_2\ln M_2$ where
 $\overline{H}^{M_1,M_2}_\theta(u_1, u_2):=$
  \be\label{Hdef}\frac{\alpha}{2}\int_\Omega| \nabla u_1|^2 - \frac{1}{2}\theta\gamma \int_\Omega |\nabla u_2|^2  -\beta\int_\Omega \nabla u_1\cdot\nabla u_2
 -M_1\ln\left( \int e^{\alpha u_1-\beta u_2}\right)-\theta M_2\ln\left( \int e^{-\gamma u_2-\theta\beta u_1}\right)  \ . \ee
  Then (\ref{qrsexp}) takes the form
 \be\label{qrs1}  \eps\frac{\partial u_1}{\partial t} = \frac{1}{\beta^2+\alpha\gamma\theta}\left( \beta\frac{\delta \overline{H}_\theta}{\delta u_2}-
 \theta\gamma\frac{\delta \overline{H}_\theta}{\delta u_1}\right)
 \ \ \ , \ \  \eps\frac{\partial u_2}{\partial t} =\frac{1}{\beta^2+\alpha\gamma\theta}\left( \beta\frac{\delta \overline{H}_\theta}{\delta u_1}
 +\alpha \frac{\delta \overline{H}_\theta}{\delta u_2}\right) \ . \ee
\end{description}
\subsection{Steady states}
Any critical point of $H_\theta$ in $\Gamma_{M_1}\times \Gamma_{M_2}\times (\mathbb{H}_0^1)^2$ is also an equilibrium solution of (\ref{sys1}, \ref{sens0}).
The variation of $H_\theta$ with respect to $\rho_i$ yields
  $$\ln\rho_1-\alpha u_1+\beta u_2=\lambda_1 \ \ \ ; \ \ \ \theta( \ln\rho_2+\gamma u_2)+\beta u_1=\lambda_2$$
 where $\lambda_i$ are the Lagrange multipliers associated with the constraints $\int\rho_i=M_i$.
 Hence
  \be\label{ls0}\rho_1=M_1\frac{e^{\alpha u_1-\beta u_2}}{\int_\Omega e^{\alpha u_1-\beta u_2}} \ \ \ ; \ \ \ \rho_2=M_2\frac{e^{-\gamma u_2-\theta\beta u_1}}{\int_\Omega e^{-\gamma u_2-\theta\beta u_1}}  \ . \ee
 The variation of $H_\theta$ with respect to $(u_1,u_2)\in (\mathbb{H}_0^1)^2$ yields
 \be\label{us} \rho_i=-\Delta u_i \ \ .   \ee
 Combining (\ref{ls0}, \ref{us}) together we obtain the
 {\it Liouville} type system
 \be\label{ls}\Delta u_1+M_1\frac{e^{\alpha u_1-\beta u_2}}{\int_\Omega e^{\alpha u_1-\beta u_2}}=0 \ \ \ ; \ \ \ \Delta u_2+M_2\frac{e^{-\gamma u_2-\theta\beta u_1}}{\int_\Omega e^{-\gamma u_2-\theta\beta u_1}} =0 \  , \  \ \ u_1=u_2=0 \ \ \text{on} \ \partial\Omega \ .  \ee
 It can be verified directly that a solution $(\rho_i,u_i)$ of (\ref{us}, \ref{ls}) yields a steady state solution of (\ref{sys1}, \ref{sens0}).
\begin{proposition}\label{propcrit} $(u_1,u_2)$ is a solution of the Liouville system (\ref{ls}) iff either $(u_1,u_2)$ is a critical point of  $\overline{H}^{M_1,M_2}_\theta$ in $(\mathbb{H}_0^1)^2$ or $(\rho_1,\rho_2)$, $\rho_i=-\Delta u_i$ is a critical point of $\underline{H}_\theta$ in $\Gamma_{M_1}\times \Gamma_{M_2}$ or $(\rho_1, u_2)$ is a critical point of $F_\theta^{M_2}$ in $\Gamma_{M_1}\times \mathbb{H}_0^1$.
\end{proposition}
\section{The Case of Conflict}\label{seccoc}
From now on we assume the case of conflict $\theta=-1$.
 The Liouville system (\ref{ls}) takes the form
 \be\label{lstheta-1}\Delta u_1+M_1\frac{e^{\alpha u_1-\beta u_2}}{\int_\Omega e^{\alpha u_1-\beta u_2}}=0 \ \ \ ; \ \ \ \Delta u_2+M_2\frac{e^{-\gamma u_2+\beta u_1}}{\int_\Omega e^{-\gamma u_2+\beta u_1}} =0 \  , \  \ \ u_1=u_2=0 \ \ \text{on} \ \partial\Omega \ .  \ee

 Here and thereafter we omit the index $\theta$ from $\overline{H}_\theta$, $\underline{H}_\theta$ and $F^M_\theta$. In particular
   \be\label{F^M}F^{M}(\rho,w):= \int\rho\ln\rho +\frac{\alpha}{2}\int \rho \Delta^{-1}(\rho) +\left[ \frac{\gamma}{2}\int|\nabla w|^2 + M\ln\left(\int e^{-\gamma w-\beta \Delta^{-1}(\rho)}\right) \right]\ee

 \begin{lemma}
 $$\sup_{\rho_2\in\Gamma_{M_2}}\underbar{H}(\rho_1,\rho_2)= \inf_{w\in \mathbb{H}_0^1} F^{M_2}(\rho,w)- M_2\ln M_2 \ . $$
 \end{lemma}
 \begin{proof}
 First note that
 $$ \frac{1}{2}\int \rho \Delta^{-1}(\rho) = \inf_{w\in \mathbb{H}_0^1}\frac{1}{2}\int|\nabla w|^2 + \int \rho w \ . $$
 $$ \underline{H}(\rho_1, \rho_2):= \int\rho_1\ln\rho_1 - \int\rho_2\ln\rho_2 + \frac{\alpha}{2}\int\rho_1 \Delta^{-1}(\rho_1) + \frac{\gamma}{2}\int\rho_2 \Delta^{-1}(\rho_2)-\beta\int\rho_2\Delta^{-1}(\rho_1) \ . $$
 $$ \leq  \int\rho_1\ln\rho_1 - \int\rho_2\ln\rho_2 + \frac{\alpha}{2}\int\rho_1 \Delta^{-1}(\rho_1) + \gamma\left( \frac{1}{2}\int|\nabla w|^2 - \int \rho_2 w\right)-\beta\int\rho_2\Delta^{-1}(\rho_1):= \underline{H}(\rho_1,\rho_2,w) \ , $$
 and $\inf_{w\in \mathbb{H}_0^1}\underline{H}(\rho_1,\rho_2, w)=\underline{H}(\rho_1,\rho_2)$.
 A direct calculation shows that
 $$  \sup_{\rho_2\in\Gamma_{M_2}}\left\{ - \int\rho_2\ln\rho_2  - \gamma \int \rho_2 w-\beta\int\rho_2\Delta^{-1}(\rho_1)\right\}
 = M_2\ln\left(\int e^{-\gamma w-\beta \Delta^{-1}(\rho_1)}\right)-M_2\ln M_2 \ . $$
 In particular,
 $$ \sup_{\rho_2\in\Gamma_{M_2}}\underline{H}(\rho,\rho_2,w)=F^{M_2}(\rho,w)-M_2\ln M_2 \ . $$
 Hence $\inf_{w\in \mathbb{H}_0^1}F^{M_2}(\rho,w)-M_2\ln M_2=$
 $$ \inf_{w\in \mathbb{H}_0^1}\sup_{\rho_2\in\Gamma_{M_2}}\underline{H}(\rho,\rho_2,w)=
 \sup_{\rho_2\in\Gamma_{M_2}}\inf_{w\in \mathbb{H}_0^1}\underline{H}(\rho,\rho_2,w)
 =\sup_{\rho_2\in\Gamma_{M_2}}\underline{H}(\rho,\rho_2)$$
 \end{proof}
 \begin{lemma}\label{lemma1.2}
 For any $(u_1,u_2)\in (\mathbb{H}_0^1)^2$
  \begin{multline}\label{oH=H}\overline{H}^{M_1,M_2}(u_1,u_2)+M_1\ln M_1-M_2\ln M_2 =\\ \inf_{\rho_1\in\Gamma_{M_1}}\sup_{\rho_2\in\Gamma_{M_2}} H(\rho_1,\rho_2, u_1, u_2) \equiv \sup_{\rho_2\in\Gamma_{M_2}}\inf_{\rho_1\in\Gamma_{M_1}} H(\rho_1,\rho_2, u_1, u_2) \ . \end{multline}
 If, in addition,  $\alpha\gamma\geq \beta^2$  then for any $(\rho_1, \rho_2)\in\Gamma_{M_1, M_2}$
 \be\label{uH=H} \underline{H}(\rho_1,\rho_2)= \inf_{u_1,u_2\in \mathbb{H}_0^1} H(\rho_1, \rho_2, u_1, u_2) \ . \ee
 \end{lemma}
 \begin{proof}
 The  equalities  in (\ref{oH=H}) follow from the definition of $\overline{H}$  (\ref{Hdef}) which use (\ref{ls0}). Indeed, (\ref{ls0}) are the {\it unique} minimizer (maximizer) of $H$ as a function of $\rho_1$ ($\rho_2$) where $u_1, u_2$ are fixed, since $H$ is strictly convex in $\rho_1$ and strictly concave in $\rho_2$.
 \par
 To get (\ref{uH=H}) note that  $\alpha\gamma > \beta^2$ implies that $H$ is strictly convex, jointly  in $(u_1,u_2)$, and the only minimizer is $\Delta u_i+\rho_i=0$, namely $u_i=-\Delta^{-1}(\rho_i)$, $i=1,2$. Then (\ref{uH=H}) follows directly from definition (\ref{uHdef}).  The case of equality $\alpha\gamma= \beta^2$ follows from a simple limit argument.
 \end{proof}
 \begin{lemma}\label{lemmaoverH}
 If $\alpha\gamma\geq \beta^2$ then
 $$\inf_{u_1,u_2\in \mathbb{H}_0^1}\overline{H}^{M_1,M_2}(u_1,u_2)= \inf_{\rho\in\Gamma_{M_1}; w\in \mathbb{H}_0^1} F^{M_2}(\rho,w) - M_2\ln M_2 \ . $$
 \end{lemma}
 \begin{proof}
Since $H$ is jointly convex in $(u_1, u_2)$ and concave in $\rho_2$ it follows, by the minmax Theorem,  that  \be\label{suprho2u1u2}\sup_{\rho_2\in\Gamma_{M_2}}\inf_{u_1,u_2\in \mathbb{H}_0^1} H(\rho_1, \rho_2, u_1, u_2)=\inf_{u_1,u_2\in \mathbb{H}_0^1}\sup_{\rho_2\in\Gamma_{M_2}}H(\rho_1, \rho_2, u_1, u_2)\ . \ee
 So
 \begin{multline}\label{longineq} \inf_{\rho_1\in\Gamma_{M_1}}\sup_{\rho_2\in\Gamma_{M_2}}\underline{H}(\rho_1,\rho_2)= \\ \inf_{\rho_1\in\Gamma_{M_1}}\sup_{\rho_2\in\Gamma_{M_2}}\inf_{u_1,u_2\in \mathbb{H}_0^1} H(\rho_1, \rho_2, u_1, u_2)=\inf_{\rho_1\in\Gamma_{M_1}}\inf_{u_1,u_2\in \mathbb{H}_0^1}\sup_{\rho_2\in\Gamma_{M_2}}H(\rho_1, \rho_2, u_1, u_2) \\
 =\inf_{u_1,u_2\in \mathbb{H}_0^1}\inf_{\rho_1\in\Gamma_{M_1}}\sup_{\rho_2\in\Gamma_{M_2}}H(\rho_1, \rho_2, u_1, u_2)= \inf_{u_1,u_2\in \mathbb{H}_0^1}\overline{H}^{M_1, M_2}(u_1, u_2) \ , \end{multline}
 where the first equality from (\ref{uH=H}), the second one from (\ref{suprho2u1u2}), the third one is trivial and the last one follows from
 (\ref{oH=H}).
 \end{proof}
  \begin{lemma}\label{lemmaFM}
  If  $\alpha\gamma>\beta^2$ then $\overline{H}^{M_1,M_2}$ is a Lyapunov functional for (\ref{qrs1}), that is
  $$ \frac{d}{dt}\overline{H}^{M_1,M_2}(u_1(\cdot, t), u_2(\cdot,t))\leq  0$$
  where $(u_1,u_2)$ is a solution of (\ref{qrs1}) in $C^1\left(\R^+; \left(\mathbb{H}_0^1(\Omega)\right)^2\right)$. The above equality is strict unless $(u_1,u_2)$ is a steady state of this system.
  \end{lemma}
  \begin{proof} of Lemma\ref{lemmaFM} \\
From (\ref{qrs}) we get
 \begin{multline}\frac{d}{dt}\overline{H}^{M_1,M_2}=\delta_{u_1}\overline{H}^{M_1, M_2}\frac{\partial u_1}{\partial t}+ \delta_{u_2}\overline{H}^{M_1, M_2}
  \frac{\partial u_2}{\partial t}\\
 =  -\frac{\eps}{\alpha\gamma-\beta^2}\left[ \left(\beta\delta_{u_2}\overline{H}^{M_1,M_2} + \gamma\delta_{u_1}\overline{H}^{M_1,M_2}\right)\delta_{u_1}\overline{H}^{M_1, M_2}\right. \\ \left.
  + \left( \beta\delta_{u_1}\overline{H}^{M_1,M_2} + \alpha\delta_{u_2}\overline{H}^{M_1,M_2}\right)\delta_{u_2}\overline{H}^{M_1, M_2}\right]
  \\ = -\frac{\eps}{\alpha\gamma-\beta^2}\left[ \gamma \| \delta_{u_1} \overline{H}^{M_1, M_2}\|_2^2
 +\alpha \| \delta_{u_2} \overline{H}^{M_1, M_2}\|_2^2 +2\beta \left< \delta_{u_1}\overline{H}^{M_1,M_2}, \delta_{u_2}\overline{H}^{M_1,M_2}\right> \right]\\
 \leq -\frac{\eps}{\alpha\gamma-\beta^2}\left[ \gamma \| \delta_{u_1} \overline{H}^{M_1, M_2}\|_2^2
 +\alpha \| \delta_{u_2} \overline{H}^{M_1, M_2}\|_2^2 -2\beta  \|\delta_{u_1}\overline{H}^{M_1,M_2}\|_2 \|\delta_{u_2}\overline{H}^{M_1,M_2}\|_2 \right]
 \end{multline}
%
   where we used  Cauchy-Schwartz inequality. Since the quadratic form in $\delta_{u_i}\overline{H}^{M_1, M_2}$ is positive definite, the last term is non-positive and, in fact, negative unless  $\delta_{u_i}\overline{H}^{M_1, M_2}=0$ for \\ $i=1,2$.
\end{proof}
Let
\be\label{barF}\bar{F}^M(\rho):=\inf_{w\in \mathbb{H}^1_0}F^M(\rho,w) \  . \ee
\begin{definition}\label{defGamma}
Let $M_1>0, M_2\geq 0$. $(M_1,M_2)\in \underline{\Lambda}$ if and only if  $\bar{F}^{M_2}$ is unbounded from below on $\Gamma_{M_1}$.
The set where $\bar{F}^{M_2}$ is bounded from below on $\Gamma_{M_1}$ is $\overline{\Lambda}$ .
\par
In the case where $\Omega$ is a disc $D_R:=\{|x|\leq R\}$ we denote $\Gamma^R_{M}\subset \Gamma_{M}(D_R)$ the set of all radial functions in $\Gamma_{M}(D_R)$. Then $\underline{\Lambda}^R$ (resp. $\overline{\Lambda}^R$) is defined as above for $\bar{F}^{M_2}$ restricted to $\Gamma^R_{M_1}$.
\end{definition}
\begin{proposition}\label{pro1}
If $(M_1,M_2)$ is an interior point of $\overline{\Lambda}$  then there exists a minimizer of $F^{M_2}$ on $\Gamma_{M_1}\times \mathbb{H}_0^1$. If, moreover, $\alpha\gamma>\beta^2$ then this minimizer induces  a minimizer of $\overline{H}^{M_1,M_2}$
on $(\mathbb{H}_0^1)^2$ as well.
\end{proposition}
\begin{proof}
Let $q\in (0,1)$, $\tilde{\gamma}:=\gamma q$, $\tilde{\beta}:=\beta/q$. Let  $\overline{\tilde{\Gamma}}$  defined according to Definition  \ref{defGamma}  with respect to $\tilde{F}$, where
$$\tilde{F}^M(\rho,w):= \int\rho\ln\rho +\frac{\alpha}{2}\int\rho \Delta^{-1}(\rho) + \frac{\tilde{\gamma}}{2}\int|\nabla w|^2 + M\ln\left(\int e^{-\tilde{\gamma} w-\tilde{\beta} \Delta^{-1}(\rho)}\right) \ . $$  We can find such $q$ for which  $(\tilde{M}_1,\tilde{M}_2):= (qM_1, q^{-1}M_2)\in \overline{\tilde{\Gamma}}$.  Set $\rho:=q\tilde{\rho}$, $w=q\tilde{w}$.  Then
$$ \tilde{F}^{\tilde{M}_2}(\tilde{\rho},\tilde{w}):= q^{-1}\int\rho\ln\rho+\frac{\alpha}{2q^2}\int \rho \Delta^{-1}(\rho) + \frac{\tilde{\gamma}}{2q^2}\int|\nabla w|^2 + \tilde{M}_2\ln\left(\int e^{-(\tilde{\gamma}/q) w-\tilde{\beta} q^{-1}\Delta^{-1}(\rho)}\right) $$
$$+q^{-1}M_1\ln q $$
$$ = q^{-1}\left[ \int\rho\ln\rho +\frac{\alpha}{2q}\int\rho \Delta^{-1}(\rho) + \frac{\gamma}{2}\int|\nabla w|^2 + M_2\ln\left(\int e^{-\gamma w-\beta \Delta^{-1}(\rho)}\right)\right]+q^{-1}M_1\ln q $$
$$= q^{-1}\left[ F^{M_2}(\rho,w) -\frac{\alpha}{2}(1-q^{-1}) \int \rho \Delta^{-1}(\rho)\right] +q^{-1}M_1\ln q \ . $$
Since $\tilde{F}^{\tilde{M}_2}(\tilde{\rho},\tilde{w})>C$ for some $C\in\R$ independent of $\tilde{\rho},\tilde{w}\in \Gamma_{\tilde{M}_1}\times \mathbb{H}_0^1$   by assumption, it follows
\be\label{FM2bound} F^{M_2}(\rho,w)\geq qC-M_1\ln q + \frac{\alpha}{2}(1-q^{-1}) \int\rho \Delta^{-1}(\rho) \ee
for any $(\rho,w)\in\Gamma_{M_1}\times \mathbb{H}_0^1$.
Let now $(\rho^n, w^n)$ be a minimizing sequence for $F^{M_2}$ in $\Gamma_{M_1}\times \mathbb{H}_0^1$. From (\ref{FM2bound}), and since $q\in(0,1)$ we conclude that $\int \rho^n \Delta^{-1}(\rho^n)$ is bounded uniformly from below. Since
$$ \frac{\gamma}{2}\int|\nabla w^n|^2 +M_2\ln\left(\int e^{-\gamma w^n-\beta \Delta^{-1}(\rho^n)}\right)\geq \frac{\gamma}{2}\int|\nabla w^n|^2 +M_2\ln\left(\int e^{-\gamma w^n}\right)$$
is bounded from below as well, we obtain that $\int \rho^n\ln\rho^n$ is  bounded from above. Let $\bar{\rho}$ be a weak limit of $\rho^n$ in the Zygmund space $\mathbb{L}\ln\mathbb{L}$. Then
$$ \int\bar{\rho}\ln\bar{\rho}\leq \liminf_{n\rightarrow\infty}\int\rho^n\ln\rho^n \ . $$
On the other hand, $\Delta^{-1}$ is a compact operator from  $\mathbb{L}\ln\mathbb{L}$ to its dual space $\mathbb{L}_{\exp}$, composed of all functions $w$ for which $e^{\lambda|w|}$ is integrable for some $\lambda(w)>0$. Hence $v^n:= -\Delta^{-1}(\rho^n)$ admits a strongly convergent subsequence in $\mathbb{L}_{\exp}$, whose limit is $\bar{v} =-\Delta^{-1}(\bar{\rho})$.
Hence
$$\lim_{n\rightarrow\infty}\int \rho^n \Delta^{-1}(\rho^n)= \int\bar{\rho} \Delta^{-1}(\bar{\rho})  \ . $$
Observe that $w^n$ are uniformly bounded in the $\mathbb{H}^1$ norm. Let $\bar{w}$ be its weak limit. By embedding of $\mathbb{H}^1_0$ in  $\mathbb{L}_{exp}$  we also obtain that   $e^{-\gamma w^n}$ is a strongly convergent sequence in $\mathbb{L}^p$ for any $p<\infty$, and its limit is
 $e^{-\gamma\bar{w}}$. This, and the strong convergence of $v^n$ in $\mathbb{L}_{exp}$ imply that
 $$\lim_{n\rightarrow\infty}\int e^{-\gamma w^n-\beta \Delta^{-1}(\rho^n)}= \int e^{-\gamma \bar{w}-\beta \Delta^{-1}(\bar{\rho})} \ . $$
as well as
$$ \liminf_{n\rightarrow\infty}\int|\nabla w^n|^2\geq \int|\nabla \bar{w}|^2 \ . $$

In particular it follows that
$$ \inf_{\rho\in\Gamma_{M_1}, u\in \mathbb{H}_0^1}  F^{M_2}(\rho,u)= F^{M_2}(\bar{\rho}, \bar{w}) \ . $$
\par
The proof for $\overline{H}^{M_1,M_2}$ in case $\alpha\gamma>\beta^2$ is easier, and is left to the reader.
\end{proof}
\subsection{Objectives}\label{secobj}
Our object is to characterize the sets $\overline{\Lambda}$ and $\underline{\Lambda}$.

Let
 \be\label{newH1}H^M_\gamma(\rho,w):=\frac{\gamma}{2}\int|\nabla w|^2 + M\ln\left(\int e^{-\gamma w-\beta \Delta^{-1}(\rho)}\right)  \ , \ \   \bar{H}^M_\gamma(\rho)=\inf_{w\in\mathbb{H}_0^1}H^M_\gamma(\rho,w) \  \ee
and recall $F$ given by  (\ref{IF}).   By (\ref{F^M}, \ref{newH1}) we get
$$ F^M(\rho,w)=F(\rho)+ H_\gamma(\rho,w)$$
 and by (\ref{barF}, \ref{newH1})
\be\label{newF}  \bar{F}^{M}(\rho) = F(\rho)+ \bar{H}^{M}_\gamma(\rho)\ee

 Note that $\bar{H}^{M}_\gamma$ is bounded from below uniformly in $\rho\in\Gamma_{M_1}$.   Indeed, since $\Delta^{-1}(\rho)\leq 0$ by the maximum principle, it follows that
  $$H_\gamma^M(\rho,w)\geq H_\gamma^M(0,w) = \frac{\gamma}{2}\int|\nabla w|^2 + M\ln\left(\int e^{-\gamma w}\right) \   $$
  for any $w\in \mathbb{H}_0^1$. The last expression is bounded from below on $\mathbb{H}_0^1$ for any $M>0$.

  Hence $\bar{F}^{M_2}(\rho)$ is bounded from below whenever $F$ is. The Free energy $F$  is bounded from below on $\Gamma_{M_1}$ for $M_1\leq 8\pi/\alpha$.  Hence, we expect that $\overline{\Lambda}$ contains  $M_1\leq 8\pi/\alpha$ for {\it any} $M_2\geq 0$ (recall Definition \ref{defGamma}).

To evaluate $\underline{\Lambda}$ we only have to indicate a sequence $\rho_j\in\Gamma_{M_1}$ for which $\bar{F}^{M_2}(\rho_j)\rightarrow-\infty$.  It is enough to establish such a sequence of radial functions in the disc, i.e in $\overline{\Lambda}^R$. The evaluation of $\overline{\Lambda}$ is more subtle. At this stage we can only investigate $\overline{\Lambda}^R$.

\subsection{Main results}\label{Mainresults}
Let
\be\label{gammadef}\Lambda(M_1, M_2):=  2(M_1-M_2)-\frac{\alpha  M_1^2}{4\pi}+\frac{\beta M_1M_2}{2\pi} -\gamma\frac{M_2^2}{4\pi} \ . \ee
 Note that $$\Lambda(M_1, M_2)=M_2 \left(-2+\frac{\beta M_1-\gamma M_2}{2\pi}\right) +  \left(2M_1-\frac{\alpha M_1^2}{4\pi}+\frac{\gamma M_2^2}{4\pi}\right):= \Lambda_1(M_1, M_2) + \Lambda_2(M_1, M_2) \ . $$

  \begin{theorem}\label{blow}
 If  both $\Lambda(M_1, M_2)<0$ and   $\Lambda_2(M_1, M_2)<0$ then $(M_1, M_2)\in\underline{\Lambda}$.
  \end{theorem}
\begin{theorem}\label{radbd}
If   $M_1\leq 8\pi /\alpha$ then $(M_1, M)\in\overline{\Lambda}$ for any $M_2>0$.
Let  the  disc $D_R:= \{|z|\leq R\}$ be  our domain, for some $R>0$, and $\overline{\Lambda}^R$ as  in Definition \ref{defGamma}.  
Assume
\begin{description}
 \item{a)}$\Lambda(M_1, M_2)>0$ and  $2\beta/\alpha>\gamma M_2/4\pi+1$.
 \item{b)}  $(M_1,M_2)$ satisfies (a) and $M\geq M_2$
 \end{description}
 then $(M_1, M)\in\overline{\Lambda}^R$. 
\end{theorem}
We now show that Theorem \ref{radbd} implies Theorem \ref{T1}. 

Note that $\Lambda_1(8\pi/\alpha,M_2)>0$ if and only if $2\beta/\alpha>\gamma M_2/4\pi+1$ and $M_2>0$. 
In that case,   $\Lambda_1(M_1,M_2)=0$ intersects $M_1=8\pi/\alpha$  at  $M_2:=\frac{4\pi}{\gamma} \left(\frac{2\beta}{\alpha}-1\right)>0$ (in particular, since  $2\beta/\alpha>1$). If $\gamma=0$ then there is no intersection. In any case, the domain  where  both $\Lambda(M_1, M_2)>0$ and $2\beta/\alpha>\gamma M_2/4\pi+1$ is contained in the strip $M_1 < \underline{M}$ as defined in  (\ref{barMdef}). By  part (b) of  the Theorem \ref{radbd} we observe that, indeed, $\overline{\Lambda}^R$ contains  the domain above the lower branch of $\Lambda=0$ in that strip.

 In the case $\gamma=0$, Theorems \ref{blow} and \ref{radbd} give an (almost) complete description. Indeed, in that case $\Lambda_2(M_1, M_2)=0$ iff $M_1=8\pi/\alpha$, so the conditions of Theorems \ref{blow} and \ref{radbd} are complementary.
\section{Proofs}\label{secprof}
Without any limitation of  generality we may assume that $\Omega :=D$ is the unit disk  $\{|x|\leq 1\}$.  Denote $(f,g):=\int_{D} f(x)g(x)dx$ for a pair of integrable functions $f,g$ on $D$.
\begin{proof} of Theorem \ref{blow}: \\
Let
  $\rho=\rho(|z|)$. For $\psi\geq 1$ set $\rho^\psi(r)= \psi^2\rho(\psi r)$ if $r\in[0, 1/\psi]$, $\rho^\psi=0$ if $r\in (1/\psi, 1]$.   Then, for $w\in\mathbb{H}_0^1$ and $M>0$:
$$w_M^\psi(r):=\left\{\begin{array}{cc}
w(\psi r)-(2\pi)^{-1}M\ln (1/\psi) &   \ \text{if} \ 0\leq  r\leq 1/\psi   \\
 -(2\pi)^{-1} M\ln (r) &   \  \text{if} \ 1\geq  r\geq 1/\psi
                                  \end{array}\right.$$

             Note that
under this scaling $\Delta^{-1} \rho^\psi=-w_M^\psi$ if $\Delta w=-\rho$ and $\rho\in\Gamma_M$.  Also  $w_M^\psi\in \mathbb{H}_0^1$ for any $M>0$  if $w\in \mathbb{H}_0^1$.
We obtain for $\rho\in\Gamma_{M_1}$
\be\label{entropalpha}\int_{D}\rho^\psi\ln\rho^\psi = 2M_1\ln\psi + \int_{D}\rho\ln\rho \ . \ee
\be\label{ekalpha}(\Delta^{-1}\rho^\psi, \rho^\psi ) = (\Delta^{-1}\rho, \rho) -\frac{M_1^2}{2\pi} \ln\psi  \ , \ee
and for $w\in\mathbb{H}_0^1$:
\be\label{w} \int_{D} |\nabla w_{M_2}^\psi|^2 = \int_{D} |\nabla w|^2+ \frac{M^2_2}{2\pi}\ln\psi \ . \ee
In addition
\begin{multline}\label{logalpha} \int_{D} e^{\beta u_{M_1}^\psi-\gamma w_{M_2}^\psi}= 2\pi \left(e^{(\gamma M_2-\beta M_1)\ln(1/\psi)/2\pi}\int_0^{1/\psi}r e^{\beta u(\psi r)-\gamma w(\psi r)}+\int_{1/\psi}^1 r^{1+(\gamma M_2-\beta M_1) /2\pi}dr\right) \\
= \left[2\pi \int_0^1r e^{\beta u( r)-\gamma w(r) }+ O(1)\right]\psi^{-2+(\beta M_1-\gamma M_2)/2\pi }+ O(1)\end{multline}
It follows from (\ref{F^M}, \ref{entropalpha}-\ref{logalpha}) that if $-2+(\beta M_1-\gamma M_2)/(2\pi)> 0$ (i.e. $\Lambda_1(M_1, M_2)> 0$) then
\begin{multline}\label{Falpha} F^{M_2}(\rho^\psi, w_{M_2}^\psi)=F^{M_2}(\rho, w) + O(1)+ \\   \left[ 2(M_1-M_2)-\frac{\alpha M_1^2}{4\pi}+\frac{\beta M_2M_1}{2\pi}-\frac{\gamma M_2^2}{4\pi}\right]\ln\psi
\equiv F^{M_2}(\rho,w) + O(1)+ \Lambda(M_1, M_2)\ln\psi\end{multline}
 Letting $\psi\rightarrow\infty$ we obtain a blow-down sequence  for $F^{M_2}(\rho^\psi, w^\psi)$ where $(\rho^\psi, w^\psi)\in\Gamma_{M_1}\times \mathbb{H}_0^1$, provided $\Lambda(M_1, M_2)<0$.

If, on the other hand,  $\Lambda_1(M_1, M_2) <0$ then $F^{M_2}(\rho^\psi, w^\psi_{M_2})=$
\be F^{M_2}(\rho,w) +\left(2M_1-\frac{\alpha M_1^2}{4\pi}+\frac{\gamma M_2^2}{4\pi}\right)\ln\psi + O(1)\equiv
F^{M_2}(\rho,w)+ \Lambda_2(M_1, M_2)\ln\psi+O(1) \label{Fbeta}\ee
and the same holds if $\Lambda_2(M_1, M_2)<0$.


\end{proof}
For the proof of Theorem \ref{radbd} for the case $\gamma>0$ we shall need the following auxiliary lemma:

 \begin{lemma}\label{lemw}
For $\psi\in(0,1)$, $\gamma>0$, let  $v=v(r)$  be a solution of
 \be\label{veq} r^{-1}(rv_r)_r+r^{-\beta M/2\pi}e^{-\gamma v}=0 \ \ ;  \ \ \psi\leq r\leq 1\ee
  satisfying $v_r\leq 0$ on the interval $[\psi, 1]$ and
 $v_r(1)=-M_2/2\pi$.
 If $\frac{ \beta M-\gamma M_2}{2\pi} -2>0$ then
 $$\lim_{\psi\rightarrow 0}\ln^{-1}\left(\frac{1}{\sqrt{\psi}}\right)\int_{\sqrt{\psi} }^1r\left| v_r\right|^2 dr = \left(\frac{M_2}{2\pi}\right)^2 \ . $$
  \end{lemma}
  \begin{remark}\label{remgamma0}
  Note that we cannot give up the condition $\gamma>0$. Indeed, if $\gamma=0$ then the solution of (\ref{veq}) does not satisfy $v_{r}\leq 0$ on $[\psi,1]$ under the stated condition $\beta M>4\pi$, if $\psi>0$ is small enough.
  \end{remark}
  \begin{proof}

 Under the change of variables:
 $r\rightarrow e^{-t}$ we get that (\ref{veq}) is transformed to
 $$ \hat{v}_{tt}+ e^{(\beta M/ 2\pi -2)t -\gamma\hat{v}}=0 \ \ \ , \ \ \ 0\leq t\leq \ln(1/\psi)$$
  for $\hat{v}(t):=v(e^{-t})$.  The end point $r=1$  are transformed into $t=0$ and
  $$ \hat{v}_t(0)= M_2/ 2\pi $$
 Setting now \be\label{barhat}\bar{v}(t):= \hat{v}(t) -\gamma^{-1} (\beta M/(2\pi) -2)t\ee
  we get
\be\label{vbar} \bar{v}_{tt}+ e^{-\gamma\bar{v}}=0\ee
and
  \be\label{yomhadash} \bar{v}_t(0)=-\gamma^{-1}\left[\frac{ \beta M-\gamma M_2}{2\pi} -2\right]  \ . \ee
    From (\ref{vbar}) it follows that $|\bar{v}_t|^2/2 -\gamma^{-1}e^{-\gamma\bar{v}} :=E$ is an invariant, so
  $$ \bar{v}_t=\pm\sqrt{2(E+ \gamma^{-1}e^{-\gamma\bar{v}})} \ . $$
  for some constant $E$.  The assumption $\frac{ \beta M-\gamma M_2}{2\pi} -2>0$ and (\ref{yomhadash}) imply the $-$ sign above, so
  \be\label{bu}   \bar{v}_t=-\sqrt{2(E+ \gamma^{-1}e^{-\gamma\bar{v}})} \ee
%
  The exact solution of (\ref{bu}) which is defined on the interval $[0, \ln(1/\psi))$  and blows down at $t=\ln(1/\psi)$ is
  \be\label{exactsol} \bar{v}(t)=-\gamma^{-1}\ln(4E\gamma) -\sqrt{2E}(t+\ln(\psi)) +\frac{2}{\gamma}\ln\left(1-e^{\sqrt{2E}\gamma( t+\ln(\psi))}\right) \ . \ee
 Substitute  condition (\ref{yomhadash}) in this solution implies 
 $$\sqrt{2E} \coth \left(\sqrt{\frac{E}{2}}\gamma \ln(1/\psi)\right)= \gamma^{-1}\left[\frac{ \beta M-\gamma M_2}{2\pi} -2\right] \ . $$
  In the limit $\psi\rightarrow 0$ we get $\sqrt{2E}\rightarrow \gamma^{-1}\left[\frac{ \beta M-\gamma M_2}{2\pi} -2\right] $. From (\ref{exactsol}) we obtain that $\bar{v}_t$ converges uniformly on $[0, \ln(1/\sqrt{\psi})]$ to
  $-\gamma^{-1}\left[\frac{ \beta M-\gamma M_2}{2\pi} -2\right]$,  as $\psi\rightarrow 0$,

  Hence, by  (\ref{barhat}), we get that $\hat{v}_t$ converges uniformly on $[0, \ln(1/\sqrt{\psi})]$ to $M_2/2\pi$.

  Returning to the variable $r=e^{-t}$, recalling $v(r):= \hat{v}(-\ln(r))$ we get that for any solution $v$ of (\ref{veq}) which satisfies the conditions of the lemma,  the function $rv_r$ converges uniformly on $[\sqrt{\psi},1]$ to $-M_2/2\pi$, as $\psi\rightarrow 0$.
  Hence
  $$\int_{\sqrt{\psi}}^1 r|v_r|^2 dr= \int_{\sqrt{\psi}}^1 r^{-1}(r|v_r|)^2 dr = (M_2/2\pi)^2\ln(1/\sqrt{\psi})+o(\ln(1/\psi)) \  $$
  at  $\psi<<1$.
\end{proof}
\begin{proof} of Theorem \ref{radbd} \\
The bound of $\bar{F}^{M_2}$ from below on $\Gamma_{M_1}$ for $M_1\leq 8\pi/\alpha$ follows from the bound from below of $F$ on $\Gamma_{8\pi/\alpha}$. See (\ref{newF}) and the discussion below (\ref{IgradF0}). This concludes the first alternative of the Theorem. We assume, from now on, that $M_1\geq 8\pi/\alpha$. 
\par
Assume alternative (a), i.e.  $\Lambda(M_1, M_2)>0$,  $2\beta/\alpha>\gamma M_2/4\pi+1$ (so $\Lambda_1(8\pi/\alpha, M_2)>0$)  and $\gamma>0$.
 For the case $\gamma=0$ see Remark \ref{remgamma1} at the end of this proof.

   Note that  $\Lambda_2(8\pi/\alpha, M_2)\geq 0$.  Hence     $\Lambda(8\pi/\alpha, M_2):= \Lambda_1(8\pi/\alpha, M_2)+\Lambda_2(8\pi/\alpha, M_2)>0$ as well. Since $\Lambda(M_1, M_2)>0$ by assumption,   $s\rightarrow \Lambda_1(s, M_2)$ is linear and $s\rightarrow \Lambda(s, M_2)$ is concave, then  both $\Lambda(s, M_2)>0$ and $\Lambda_1(s, M_2)>0$ for $8\pi/\alpha \leq s \leq M_1$:
   \be\label{lambda1+} \Lambda_1(s, M_2)>0 \ \ \ \forall s\in[8\pi/\alpha, M_1] \ .  \ \ \ee
\par
Let $\delta:= (M_1-8\pi/\alpha )/n$ where $n$ is so large, for which
\be\label{con2} \min_{8\pi/\alpha\leq s\leq M_1} \Lambda(s, M_2)>  \frac{\alpha\delta M_1}{\pi}  \ . \ee
We shall prove that if $\bar{F}^{M_2}$ is unbounded from below
on $\Gamma^R_s$  where $s\in [8\pi/\alpha+\delta, M_1]$ then it is still unbounded from below on $\Gamma^R_{s-\delta}$.
Iterating this argument $n$ times  we obtain that $\bar{F}^{M_2}$ is unbounded from below on $\Gamma^R_{8\pi/\alpha}$ and get a contradiction.

So let $s$ in this interval and  $\{\rho_j\}\in \Gamma_s$  a blow down sequence (e.g.  $\bar{F}^{M_2}(\rho_j)<-j$). Choose $\psi_j\in (0,1)$ such that $\int_{D_{\psi_j}}\rho_j=s-\delta$.  Set
 $\underline{\rho}_j\in \Gamma_{s-\delta}$  which is the restriction of $\rho_j$ to $D_{\psi_j}:=\{ |x|\leq \psi\}$,  that is:
 $$ \urho_j(r)=\rho_j(r) \ \ \ \text{for} \ \ r\in[0,\psi_j] \ \ \ , \ \ \ \urho_j(r)=0 \ \ \ \text{for} \ \ r\in( \psi_j,1] \ . $$
 Our first step is to show
  \be\label{1estim} \bar{F}^{M_2}(\urho_j)\leq \bar{F}^{M_2}(\rho_j) + \pi e^{-1} - \frac{\alpha\delta s}{4\pi}\ln\psi_j \ . \ee
   Since the function $s\rightarrow -s\ln s$ is bounded from above by $e^{-1}$ it follows that
 \be\label{stepa}\int_{D}\urho_j\ln\urho_j-\int_{D}\rho_j\ln\rho_j =-\int_{D-D_{\psi_j}}\rho_j\ln\rho_j \leq \pi e^{-1} \ . \ee

 Since $\Delta^{-1}\leq 0$ on $D$ and $\urho_j\leq \rho_j$ then  $\Delta^{-1}\urho_j\geq \Delta^{-1} \rho_j$ (same reasoning can be applied via the maximum principle). Then, since $\beta>0$ we obtain  by (\ref{newH1})
  \be\label{stepb} \bar{H}^{M_2}_\gamma(\urho_j)  \leq \bar{H}^{M_2}_\gamma(\rho_j)  \ . \ee

  Finally, to estimate the difference  $(\urho_j, \Delta^{-1}\urho_j)-(\rho_j,\Delta^{-1}\rho_j)$ we observe

\be\label{stepc} (\rho_j, \Delta^{-1}\rho_j)-(\urho_j, \Delta^{-1}\urho_j)=  2(\rho_j-\urho_j, \Delta^{-1}\urho_j) +(\rho_j-\urho_j, \Delta^{-1}\rho_j-\Delta^{-1}\urho_j) \ . \ee
Since $\rho_j-\urho_j$ is supported in the ring $1\geq r\geq \psi_j$ we obtain from  (\ref{uug})
\be\label{stepd}2(\rho_j-\urho_j, \Delta^{-1}\urho_j) \geq \frac{s(s-\delta)}{2\pi}\ln\psi_j \ . \ee

 The second term
is the (negative) energy  due to a mass concentrated in the ring $r\in[\psi_j, 1]$.
It is maximized  if the mass is concentrated in the inner circle $r=\psi_j$. The potential  induced by the mass $\delta$  concentrated on this circle is just $\delta/(2\pi) \ln\psi_j$, so the energy is bounded from below by $\delta^2/(2\pi) \ln\psi_j$.
 Hence
  \be\label{stepe} (\rho_j, \Delta^{-1}\rho_j)-(\urho_j, \Delta^{-1}\urho_j)\geq \frac{\delta s}{2\pi}\ln\psi_j \ . \ee
  Summarizing  (\ref{stepa}-\ref{stepe}) and using (\ref{newF}) we obtain  (\ref{1estim}).
  \par
 Note that  potential $\uu_j=-\Delta^{-1}\urho_j$  satisfies
 \be\label{uug}\uu_j(r)=\frac{s-\delta}{2\pi}\ln\left(\frac{1}{r}\right) \ \ \ \text{for} \ \ 1\geq r\geq \psi_j \ . \ee
      Set  now $\urho^{\psi_j}(r)= \psi_j\urho(\sqrt{\psi_j} r)$ for $r\in[0,1]$. Note that $\urho_j^{\psi_j}$ is supported on the disc or radius $\sqrt{\psi}$.  Evidently, $\urho^{\psi_j}\in\Gamma_{s-\delta}$ as well.

  For $\urho_j^{\psi_j}$ as above,  (\ref{entropalpha}, \ref{ekalpha}) imply
\be\label{entropalpha1}\int_{D}\urho_j^{\psi_j}\ln\urho_j^{\psi_j} = 2M_1\ln\sqrt{\psi_j}+ \int_{D}\urho_j\ln\urho_j   \ee
\be\label{ekalpha1}(\Delta^{-1}\urho_j^{\psi_j}, \urho_j^{\psi_j} ) = (\Delta^{-1}\urho_j, \urho_j) -\frac{M_1^2}{2\pi} \ln\sqrt{\psi_j } \ . \ee
We obtained
\be\label{secondFestim} F(\urho_j^{\psi_j})=F(\urho_j) +\left(2M_1-\frac{\alpha M_1^2}{2\pi} \right)\ln\sqrt{\psi_j } \ . \ee
Next we estimate $\bar{H}_\gamma^{M_2}(\urho_j^{\psi_j})$ in terms of $\bar{H}_\gamma^{M_2}(\urho_j)$.

Set  $\uw_j\in\mathbb{H}^1_0(D)$ to be the solution of
\be\label{weqbar}\Delta \uw_j + \frac{M_2}{\int_D e^{\gamma \uw_j+\beta \uu_j}} e^{-\gamma \uw_j+\beta \uu_j} =0 \  . \ee
 Recall that $\uw_j$ is the minimizer of $H^{M_2}_\gamma(\urho_j,w)$ (see \ref{newH1})). In particular
 \be\label{koko} \bar{H}_\gamma^{M_2}(\urho_j)=H_\gamma^{M_2}(\urho_j, \uw_j) \ . \ee
 The function $\uw_j$  is a radial function  on the interval $[0,1]$. Indeed, it satisfies, as a function of $r$
 \be\label{rweqbar} r^{-1}(r \uw_{j,r})_r + \lambda e^{-\gamma \uw_j+\beta \uu_j} =0 \  , \ee
 where $\lambda>0$ is an appropriate constant verifying $\uw_{j,r}(1)=-M_2/(2\pi)$.  Since $\uw_{j,r}(0)=0$ it follows $\uw_{j,r}\leq 0$ on the entire interval $[0,1]$.
Using (\ref{uug}) in (\ref{weqbar}) we observe that $\uw_j$ is identified, up to an additive constant,  with a solution $v$ of  (\ref{veq}) on the annulus $\psi_j\leq r\leq 1$, where $M=s-\delta$.  Recalling (\ref{lambda1+}), we obtain from  Lemma \ref{lemw}
\be\label{ttt} \frac{\gamma}{2}\int_{\sqrt{\psi}\leq |x|\leq 1} |\nabla \uw_j|^2= \frac{\gamma M_2^2}{4\pi} \ln(1/\sqrt{\psi}) + o(|\ln\psi|) \ . \ee
The potential corresponding to $\urho_j^{\psi_j}$ is  $\uu_j^{\psi_j}(r)=\uu_j(\sqrt{\psi_j} r)+ ((s-\delta)/2\pi)\ln\sqrt{\psi_j}$. Define also
  $\uw_j^{\psi_j}(r)=\uw_j(\sqrt{\psi_j} r)-\uw_j(\sqrt{\psi_j})\in \mathbb{H}_0^1(D)$.
As in
 (\ref{logalpha} ) (with $\sqrt{\psi_j}$ replacing $\psi_j$) we obtain
\begin{multline} \int_{D} e^{\beta\uu_j^{\psi_j}-\gamma \uw_j^{\psi_j}}= 2\pi e^{(\beta (s-\delta)\ln\sqrt{\psi_j}/2\pi+\gamma \uw_j(\sqrt{\psi_j}))}\int_0^1re^{\beta\uu_j(\sqrt{\psi_j} r)-\gamma \uw_j(\sqrt{\psi_j} r)}dr \\ =2\pi e^{(\beta/(2\pi) (s-\delta)\ln\sqrt{\psi_j}+\gamma \uw_j(\sqrt{\psi_j})- 2\ln\sqrt{\psi_j})}\int_0^{\sqrt{\psi_j}}r e^{\beta\uu_j( r)-\gamma \uw_j(r)}dr \\
\leq 2\pi e^{(\beta (s-\delta)/(2\pi)\ln\sqrt{\psi_j}+\gamma \uw_j(\sqrt{\psi_j})- 2\ln\sqrt{\psi_j})}\int_0^1r e^{\beta\uu_j( r)-\gamma \uw_j(r)}dr \end{multline}
It follows that
\begin{multline}\label{logalpha1} M_2\ln\left(\int_{D} e^{\beta\uu_j^{\psi_j}-\gamma \uw_j^{\psi_j}}\right) \leq \\ M_2\ln\left(\int_{D} e^{\beta \uu_j-\gamma \uw_j}\right)
+M_2\left\{\left[ \beta (s-\delta)/(2\pi)-2\right]\ln\sqrt{\psi_j}+\gamma  \uw_j(\sqrt{\psi_j})) \right\} \ .
\end{multline}
Recall that $\uw_j$ is the solution of (\ref{rweqbar}) satisfying $\uw_{j}(1)=0, \uw_{j,r}(1)=-M_2/(2\pi)$. In particular $(r\uw_{j,r})_r<0$ on $(0,1]$.   It follows that $\uw_j(r)\leq - (M_2/2\pi) \ln(r)$ for any $r\in[0,1]$.  Hence
 $$ [\beta (s-\delta)/(2\pi)-2]\ln\sqrt{\psi_j}+\gamma w_j(\sqrt{\psi_j})) \leq  \left(\frac{\beta(s-\delta)}{2\pi} -\frac{\gamma M_2}{2\pi}-2\right)\ln\sqrt{\psi_j} \ , $$
so
\be\label{ss1}M_2\ln\left(\int_{D} e^{\beta\uu_j^{\psi_j}-\gamma \uw_j^{\psi_j}}\right) \leq  M_2\ln\left(\int_{D} e^{\beta\uu_j-\gamma \uw_j}\right) +M_2\left( \frac{\beta(s-\delta)}{2\pi} -\frac{\gamma M_2}{2\pi}-2\right)\ln\sqrt{\psi_j} \ . \ee

Next,
\begin{multline}\label{ss2} \frac{\gamma}{2}\int_{D}\left|\nabla\uw_j^{\psi_j}\right|^2 = \pi\gamma\int_0^1\left|\frac{d\uw_j^{\psi^j}}{dr}\right|^2 rdr
= \pi\gamma\int_0^1 \left| \uw_j\prime(\sqrt{\psi_j} r) \right|^2 \psi_jrdr=   \\ \pi\gamma\int_0^{\sqrt{\psi_j}} \left| \uw_j\prime( r) \right|^2 rdr=  \pi\gamma\int_0^{1} \left| \uw_j\prime( r) \right|^2 rdr- \gamma\pi\int_{\sqrt{\psi_j}}^1\left| \uw_j\prime( r) \right|^2 rdr   =  \\ \frac{\gamma}{2}\int_{D}\left|\nabla \uw_j\right|^2+\frac{\gamma M_2^2}{4\pi} \ln\sqrt{\psi_j} + o(|\ln\psi_j|) \end{multline}
where we used   (\ref{ttt}) in the last equality.

From (\ref{newH1}, \ref{koko}, \ref{ss1}, \ref{ss2}) we obtain
\begin{multline} \label{multstut}\bar{H}_\gamma^{M_2}(\urho_j^{\psi_j})\leq H_\gamma^{M_2}(\urho_j^{\psi_j}, \uw_j^{\psi_j})\leq H_\gamma^{M_2}(\urho_j, \uw_j)
+ M_2\left( \frac{\beta(s-\delta)}{2\pi} -\frac{\gamma M_2}{4\pi}-2\right)\ln(\sqrt{\psi_j})+ o(\ln(\psi_j)) \\ =\bar{H}_\gamma^{M_2}(\urho_j)
+ M_2\left( \frac{\beta(s-\delta)}{2\pi} -\frac{\gamma M_2}{4\pi}-2\right)\ln(\sqrt{\psi_j}) + o(\ln(\psi_j))\end{multline}
This and (\ref{secondFestim}), together with  the definition (\ref{gammadef}) of $\Lambda$, imply
 $$ \bar{F}^{M_2}(\urho_j^{\psi_j})\leq \bar{F}^{M_2}(\urho_j) +\left[\Lambda(s-\delta, M_2) +o(1)\right]\ln\sqrt{\psi_j} \ . $$
Using this in (\ref{1estim}):
 $$ \bar{F}^{M_2}(\urho_j^{\psi_j})\leq \bar{F}^{M_2}(\rho_j) +\left( \Lambda(s-\delta, M_2)-\frac{\alpha\delta s}{\pi} +o(1)\right)  \ln\sqrt{\psi_j} +\pi e^{-1}\ . $$
 Since $\psi_j\in (0,1)$ it follows by (\ref{con2}) that $F^{M_2}(\urho^{\psi_j}_j)\leq F^{M_2}(\rho_j) +\pi e^{-1}$, as long as $s-\delta\leq M_1$. In particular $\lim_{j\rightarrow\infty} F^{M_2}(\urho_j^{\psi^j})=-\infty$ if $\lim_{j\rightarrow\infty} F^{M_2}(\rho_j)=-\infty$.  Recalling $\rho_j\in\Gamma_s$  while $\urho_j^{\psi_j}\in\Gamma_{s-\delta}$ we obtain the desired result by $n$ iteration, as explained below (\ref{con2}).
\begin{remark}\label{remgamma1}
In the case $\gamma=0$, (\ref{ss2}) is reduced to the trivial identity $0=0$, while $H_0^{M_2}(\rho,w)$ is independent on $w$, so we may take $\uw_j=0$ and $\bar{H}_0^{M_2}(\urho_j)=H_0^{M_2}(\urho_j,0)$. The inequality (\ref{multstut}) holds with this substitution and the result follows as well. Note that we do not apply Lemma \ref{lemw} in that case.
\end{remark}
We finally turn to the proof of part (b):  By (\ref{newH1})
$$ H_\gamma^{M}(\rho,w)-H_\gamma^{M_2}(\rho,w)= (M-M_2)\ln\left( \int e^{-\gamma w-\beta\Delta^{-1}(\rho)}\right) \geq (M-M_2)\ln\left( \int e^{-\gamma w}\right)$$
since $\Delta^{-1}\rho\leq 0$ and $M\geq M_2$. If $\gamma=0$ then (\ref{newH1}, \ref{newF}) imply that $F^{M}$ is bounded from below on $\Gamma_{M_1}$ is $F^{M_2}$ is. Otherwise, Jensen, Poincare and Caushy-Schwartz inequalities imply the existence of a constant $C(\eps)>0$ such that
$$ (M-M_2)\ln\left( \int e^{-\gamma w}\right)\geq -\eps\|\nabla w\|_2^2- C(\eps)$$ 
for any $\eps>0$ and $w\in\mathbb{H}_0^1$.  Hence
$$ H^{M}_\gamma(\rho,w)\geq H^{M_2}_\gamma(\rho,w)-\eps \|\nabla w\|_2^2 -C(\eps) \  $$
for any $(\rho,w)\in \Gamma_{M_1}\times \mathbb{H}_0^1$. Scaling $w \mapsto \gamma w$ we obtain from (\ref{newH1}) 
$$ \bar{H}^M_\gamma (\rho)\geq  \bar{H}^{M_2}_{\hat{\gamma}}(\rho)-C(\eps)$$
where $\hat{\gamma}:= \frac{\gamma}{1-2\eps/\gamma}$.  From this and (\ref{newF}) we obtain that 
$\bar{F}_\gamma^{M}:= F+ \bar{H}_\gamma^M$
is bounded from below on $\Gamma_{M_1}$ if $\bar{F}^{M_2}_{\hat{\gamma}}:= F+ \bar{H}_{\hat{\gamma}}^{M_2}$ is. 
Since the conditions of (a), determined by strong inequalities, are preserved under the change $\gamma \mapsto \hat{\gamma}$ for $\eps>0$ small enough, we obtain the result. 
  \end{proof}

\begin{center}References\end{center}
\begin{enumerate}
\item\label{[B]}  Beckner, W.: {\it Sharp Sobolev inequalities on the sphere and the Moser–Trudinger inequality}. Ann. of Math. (2) 138, 213-242 (1993)
\item\label{[BCC]} Blanchet, A.,  Carlen, E.A. and  Carrillo, J.A.: {\it Functional inequalities, thick tails and asymptotic for the critical mass Patlak-Keller-Segel model} J. Funct. Anal. 262 , no. 5, 2142-2230, (2012)
\item\label{[Bi]}  Biler, P.,  Karch, G.:  {\it Blowup of solutions to generalized Keller-Segel model},  J. Evol. Equ. 10, no. 2, 247-262,  (2010)
\item\label{[CC]}  Calvez, V., Corrias, L.  {\it The parabolic-parabolic Keller-Segel model in $R^2$}, Commun. Math. Sci. 6, 417-447 (2008)
\item\label{[CL]} Carlen, E., Loss, M.: {\it Competing symmetries, the logarithmic HLS inequality and Onofri’s inequality on $S^n$}. Geom. Funct. Anal. 2, 90-104 (1992)
\item\label{[CSW]}  Chipot, M.; Shafrir, I.; Wolansky, G. {\it On the solutions of Liouville systems}. J. Differential Equations 140 , no. 1, 59-105, (1997)
\item\label{[H]} Horstmann, Dirk. {\it Generalizing the Keller-Segel model: Lyapunov functionals, steady state analysis, and blow-up results for multi-species chemotaxis models in the presence of attraction and repulsion between competitive interacting species}. J. Nonlinear Sci. 21 , no. 2, 231-270,  (2011)
   \item\label{[KS]} Keller, E. F. and  Segel, L. A.  {\it Traveling bends of chemotactic bacteria}. J. Theor. Biol. 30, 235-248 (1971)
   \item\label{[KS1]}  Kavallaris, N.,   Suzuki, T.: {\it  On the finite-time blow-up of a non-local parabolic equation describing chemotaxis}. Differential Integral Equations 20 , no. 3, 293-308, (2007)
    \item\label{[L]}   Lin, C.S: {\it Liouville systems of mean field equations}, Milan J. Math. 79, no. 1, 81-94,  (2011)
  \item\label{[RS]}    Ricciardi, T. and Suzuki, T.:  {\it Duality and best constant for a Trudinger–Moser inequality involving probability measures}. J. Eur. Math. Soc. (JEMS) 16, no. 7, 1327-1348, (2014)
  \item\label{[SW]}  Shafrir, I and  Wolansky, G.:  {\it The logarithmic HLS inequality for systems on compact manifolds.} J. Funct. Anal. 227 , no. 1, 200-226, (2005)
  \item\label{[S]}  Suzuki,T.:  {\it Free Energy and Self-Interacting Particles}, Birkh\"{a}user Boston, Boston, 2005.
   \item\label{[Wa]}  Wang, G.:  {\it Moser-Trudinger inequalities and Liouville systems}. C. R. Acad. Sci. Paris Sér. I Math. 328, no. 10, 895-900, (1999)
   \item\label{[W1]}Wolansky, G. {\it Multi-components chemotactic system in the absence of conflicts.}
European J. Appl. Math. 13 (2002), no. 6, 641-661
    \item\label{[W2]} Wolansky, G. {\it A critical parabolic estimate and application to nonlocal equations arising in chemotaxis}. Appl. Anal. 66, no. 3-4, 291-321, (1997)
 \item\label{[W3]} Wolansky, G.: {\it On the evolution of self-interacting clusters and applications to semilinear equations with exponential nonlinearity}, Festschrift on the occasion of the 70th birthday of Shmuel Agmon. J. Anal. Math. 59, 251-272, (1992)
 \item\label{[W4]} Wolansky, G.: {\it On steady distributions of self-attracting clusters under friction and fluctuations},  Arch. Rational Mech. Anal. 119, no. 4, 355-391, (1992)
\end{enumerate}
 \end{document}